\documentclass[11pt]{article}
\usepackage{amsmath,amsfonts,amsthm,epsfig,latexsym,color}
\usepackage[english,francais]{babel}
 
\makeatletter
\newcommand{\subsectionruninhead}{\@startsection{subsection}{2}{0mm}
{-\baselineskip}{-0mm}{\bf\large}}
\newcommand{\subsubsectionruninhead}{\@startsection{subsubsection}{3}{4mm}
{\baselineskip}{-4mm}{\it\normalsize}}
\makeatother

\def\AA{{\mathbb A}}

 \def\NN{{\mathbb N}}  
 \def\RR{{\mathbb R}}  \def\TT{{\mathbb T}}
   
 \def\ZZ{{\mathbb Z}}

  \def\cG{{\cal G}} \def\cM{{\cal M}} 
    \def\cT{{\cal T}}
\def\cC{{\cal C}}   \def\cO{{\cal O}} \def\cU{{\cal U}}

\def\diff{\operatorname{Diff}}
\def\dd{\operatorname{d}}
\def\B{\operatorname{B}}
\def\id{\operatorname{Id}}
\def\D{\operatorname{D}}

\newtheorem{theo}{Theorem}[section]

\newtheorem{coro}[theo]{Corollary}

\newtheorem{prop}[theo]{Proposition}

\newtheorem{claim}[theo]{Claim}

\newtheorem{rema}[theo]{Remark}

\title{Perturbation of $C^1$-diffeomorphisms and generic conservative dynamics on surfaces}
\author{Sylvain Crovisier}

\date{April 5, 2006}

\begin{document}
\selectlanguage{english}
\maketitle


\setcounter{section}{-1}
\section{Introduction}

When one studies a mechanical system with no dissipation,
the motion is governed by some ordinary differential equations which preserve
a volume form.
As an example, the forced damped pendulum with no friction gives rise to a
conservative dynamics on the open annulus (see~\cite{hubbard}):
let $\theta$ be the angle between the axis of a rigid
pendulum with the vertical and $\dot \theta$ be the angle velocity. The configurations $(\theta,\dot\theta)$ live on the infinite annulus $\AA=\RR/\ZZ \times \RR$ and the motion is governed
by the equation
$$\frac{d^2}{dt^2}\theta=-\sin(2\pi\theta)+h(t),$$
where $h$ is the forcing. The volume form $d\theta\wedge d\dot\theta$ is preserved.
By integrating the system, one obtains a global flow $(\phi_t)_{t\in \RR}$, that is a familly
of diffeomorphisms of $\AA$ which associates
to any initial configuration $(\theta,\dot\theta)$ at time $0$ the 
configuration $\phi_t(\theta,\dot\theta)$ at time $t$.
If the forcing $h$ is $T$-periodic, the flow satisfies the additional relation $\phi_{t+T}=\phi_t\circ \phi_T$
and the dynamics of the pendulum can be studied through the iterates of the conservative
annulus diffeomorphism $f=\phi_T$.
Several simple questions may be asked about this dynamics:
\begin{itemize}
\item \textit{What are the regions $U\subset \AA$ that are \emph{invariant} by $f$?}
Invariant means that $f(U)=U$.
\item \textit{Does there exist a dense set of initial data $(\theta,\dot\theta)$ which are
periodic?}
Periodic means that $f^\tau(\theta,\dot\theta)=(\theta,\dot\theta)$ for some integer $\tau\geq 1$.
\item \textit{How do the orbit separate?}
More precizely, let us consider two initial data $(\theta,\dot\theta)$ and $(\theta',\dot\theta')$ that are close.
How does the distance $d(f^n(\theta,\dot\theta), f^n(\theta,\dot\theta))$ behaves when $n$ increases?
\end{itemize}
An other example of conservative diffeomorphism is the standard map on the two-torus:
$$(x,y)\mapsto (x+y , y+ a\sin(2\pi (x+y))) \mod \ZZ^2.$$
It is sometimes considered by physicists as a model for chaotic dynamics:
the equations defining such a diffeomorphism are simple but 
we are far from beeing able to give a complete description of its dynamics.
However one can
hope that {\em some other systems, arbitrarily close to the original one, could be much easier to be described.}
To reach this goal, one has to precise what ``arbitrarily close" and ``be described" mean: the answer to our
problem will depend a lot on these two definitions. The viewpoint we adopt in this text allows to give a rather
deep description of the dynamics. However one should not forget that one can choose other definitions
which could seem also (more?) relevant and that very few results were obtained in this case.

\subsubsectionruninhead*{The setting, Baire genericity.}
In the following we consider a compact and boundaryless smooth connected surface
$M$ endowed with a smooth volume $v$ (which is, after normalization, a probability
measure) and we fix a diffeomorphism $f$ on $M$
which preserves $v$. We are aimed to describe the space of orbits of $f$ and in particular the
space of periodic orbits.

Lot of difficulties appear if one chooses an arbitrary diffeomorphism. Our philosophy
here will be to forget the dynamics which seem too pathological
hoping that the set of diffeomorphisms that we describe is large (at least dense in the
space of dynamical systems we are working with).
For us, such a set will be large if it is generic in the sense of Baire category.

This notion requires to choose carefully the space of diffeomorphisms, that should be
a Baire space: for example for any $k\in \NN$, the space $\diff_v^k$ of $C^k$-diffeomorphisms
of $M$ which preserve $v$. A set of diffeomorphisms is {\em generic} (or {\em residual}) if it contains
a dense G$_\delta$ subset of $\diff_v^k$, i.e. by Baire theorem if it contains a countable
intersection of dense open sets of $\diff_v^k$ (so that the intersection of two generic sets
remains generic).\\
In the sequel, we are interested in exhibiting
generic properties of diffeomorphisms: these are properties that are satisfied on a generic set
of diffeomorphisms. 

\subsubsectionruninhead*{An example: generic behavior of periodic orbits.}
Robinson has proven in~\cite{robinson1,robinson2} the following  generic property
which extends a previous result of Kupka and Smale to the conservative diffeomorphisms.
It is a consequence of Thom's transversality theorem.
\begin{theo}[Robinson]\label{t.periodic-generic}
When $k\geq 1$, for any generic diffeomorphism $f\in\diff_v^k$, and any periodic orbit
$p,f(p),\dots,f^\tau(p)=p$, one of the two following cases occurs (figure~\ref{f.periodic}):
\begin{itemize}
\item either the orbit of $p$ is {\em elliptic}: the eigenvalues of $\D_p f^\tau$ are non-real (in particular, this tangent map is conjugate to a rotation);
\item or $p$ is a {\em hyperbolic saddle}: the eigenvalues are real and have modulus different
from $1$. In this case, there are some one-dimensional invariant manifolds
(one stable $W^s(p)$ and one unstable $W^u(p)$) through $p$.
Points on the stable manifold converge towards the orbit of $p$ in the future, and the same for
points on the unstable manifold in the past.
\end{itemize}
\end{theo}
\begin{figure}[ht]
\begin{center}
\input{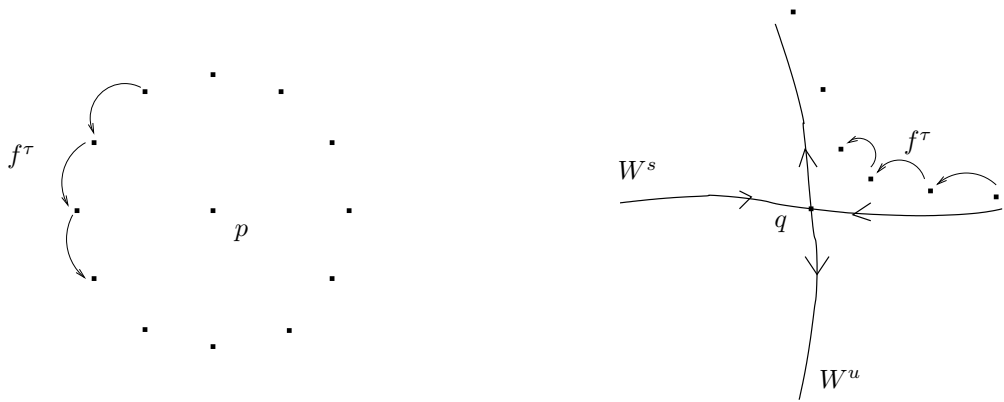}
\end{center}
\caption{Dynamics near an elliptic point, $p$, and a saddle point, $q$.\label{f.periodic}}
\end{figure}
One wants to say that the dynamics of the return map $f^\tau$ near a $\tau$-periodic $p$ ``looks like"
the dynamics of the tangent map $\D_p f^\tau$. This is the case if $p$ is hyperbolic:
D. Grobman and P. Hartman have shown that if $p$ is hyperbolic, the map $f^\tau$ is topologicall conjugate
to $\D f^\tau$ near $p$; more precizely, there exist a neighborhood $U$ of $p$
and $V$ of $0$ in the tangent space $T_pM$ and a homeomorphism $h\colon U\to V$ such that $h\circ f^\tau=\D_pf^\tau\circ h$.

Theorem~\ref{t.periodic-generic} implies in particular:
\begin{coro}\label{c.periodic-finite}
When $k\geq 1$, for any generic diffeomorphism $f\in\diff_v^k$, and any $\tau\in \NN\setminus\{0\}$,
the set of periodic points of period $\tau$ is finite.\\
\end{coro}

In the next part we will give other examples of generic properties.
They are often obtained in the same way:
\begin{itemize}
\item One first proves a \emph{perturbation result},
which is in general the difficult part.
In the previous example, one shows that for every integers $\tau,k\geq 1$, any $C^k$-diffeomorphism
$f$ can be perturbed in the space $\diff_v^k$ as a diffeomorphism $g$ whose periodic orbits of period $\tau$
are elliptic or hyperbolic.
\item One then uses Baire theorem for getting the genericity. An example of this standard argument is given at section~\ref{ss.density}.
\end{itemize}
The last two parts are devoted to some important perturbation results. In part~\ref{s.local},
we discuss Pugh's closing lemma that allows to create periodic orbits and Hayashi's closing lemma
that allows to glue two half orbits together; the perturbations in these two cases are local.
In part~\ref{s.global}, we state a connecting lemma for pseudo-orbits obtained with M.-C. Arnaud and C. Bonatti
through global perturbations and explain the main ideas of its proof.

\newpage
\renewcommand{\thesection}{\Roman{section}}
\section{Overview of genericity results on the dynamics of $C^1$ conservative surface diffeomorphisms}
This part is a survey of the properties satisfied by the $C^1$-generic
conservative diffeomorphisms of compact surfaces

\subsection{Discussions on the space $\diff_v^k$}
In general, the generic properties depend strongly on the choice of the space
$\diff_v^k$. We will here illustrate this on an example and explain why we will focus on the $C^1$-topology.

\subsubsectionruninhead{Some generic properties in different spaces.}
One of the first result was given by Oxtoby and Ulam~\cite{oxtoby-ulam}, in the $C^0$ topology.
\begin{theo}[Oxtoby-Ulam]\label{t.oxtoby-ulam}
For any generic homeomorphism $f\in \diff_v^0$, the invariant measure $v$ is ergodic.
\end{theo}
\textit{Ergodicity} means that for the measure $v$, the system can not be decomposed: any invariant Borel
set $A$ has either measure $0$ or $1$. By Birkhoff's ergodic theorem, \textit{the orbit of $v$-almost
every point is equidistributed in $M$}.\\
The $C^0$ topology also seems very weak: one can show that
since for any generic diffeomorphism, once there
exists a periodic point of some period $p$, then the set of $p$-periodic points is uncountable.
(In particular corollary~\ref{c.periodic-finite} does not hold for $\diff_v^0$.)

In high topologies, one gets Kolmogorov-Arnold-Moser theory. One of the finest forms is given
by Herman (see~\cite[section II.4.c]{moser}, \cite[chapitre IV]{herman} or \cite{yoccoz}).
\begin{theo}[Herman]\label{t.KAM}
There exists a non-empty open subset $\cU$ of $\diff_v^k$, with $k\geq 4$ and for any
diffeomorphism $f\in \cU$, there exists a smooth closed disk $D\subset M$ which is periodic
by $f$: the disks $D$, $f(D)$,\dots, $f^{\tau-1}(D)$ are disjoint and $f^\tau(D)=D$.
\end{theo}
Theses disks are obtained as neighborhoods of the elliptic periodic orbits.
The dynamics in this case is very different from the generic dynamics in $\diff_v^0$
since the existence of invariant domains breaks down the ergodicity of $v$:
the orbit of any point of $D$ can not leave the set $D\cup f(D)\cup\dots \cup f^{\tau-1}(D)$.

\begin{rema}\label{r.smooth}
We should notice that by a result of Zehnder~\cite{zehnder2} for each $k\geq 1$,
the $C^\infty$-diffeomorphisms are dense in $\diff^k_v$.
Therefore, for any $1\leq k<\ell$, any property that is generic in $\diff^\ell_v$
will be dense in $\diff^k_v$. This result is not known is this generality
in higher dimensions for conservative diffeomorphisms.
\end{rema}

\subsubsectionruninhead{An elementary perturbation lemma.}\label{ss.elementary}
The reason why theorem~\ref{t.oxtoby-ulam} is true is that perturbations in $\diff_v^0$
are very flexible: for any homeomorphism $f\in\diff_v^0$ and any point $x\in M$, one
can perturb $f$ in order to modify the image of $f(x)$.
More precisely, if $y$ is close to $f(x)$, one chooses a small path $\gamma$ that
joints $f(x)$ to $y$. Pushing along $\gamma$, one can modify $f$ as homeomorphism
$g$ so that $g(x)=y$. The homeomorphisms $f$ and $g$ will coincide outside a small
neighborhood of $\gamma$. Hence, the $C^0$-norm of the perturbation is about equal
to the distance between $x$ and $y$.

In the space $\diff_v^1$, the $C^1$-norm of the perturbation
also should be small (for example smaller than $\varepsilon>0$) and one has to perturb $f$ on a larger domain
(in a ball of radius about $\varepsilon^{-1} \dd(x,y)$). This can been seen easily
from the mean value theorem: let $x$, $y$ be two points and $\varphi$ be a perturbation of the identity which satisfies $\varphi(x)=y$
and such that $\|\D \varphi-\id\|\leq \varepsilon$; then, if a point $z$ is not perturbed by $\varphi$
(i.e. $\varphi(z)=z$), we get
$$\|y-x\|=\|(\varphi(x)-x)-(\varphi(z)-z)\|\leq \varepsilon.\|z-x\|.$$
As a consequence, when $\varepsilon$ is small, the perturbation domain
has a large radius and lot of the orbits of $f$ will be modified.

This remark will be at the root of all the genericity results that will be presented below
(a more precise statement will be given at section~\ref{ss.elementary}).
It explains the difficulty of the perturbations in $\diff_v^1$.
In higher topologies, the situation becomes much more complicate since in $\diff_v^k$, the
radius of the perturbation domain should be at least $(\varepsilon^{-1}\dd(x,y))^{\frac 1 k}$.

This justifies why we will now work in $\diff_v^1$: we need a space of {\em diffeomorphisms}
where the {\em elementary perturbations} don't have a too large support.

\subsection{The closing and connecting lemmas}
From the elementary perturbation lemma, one
derivates more sophisticated perturbation lemmas.

\subsubsectionruninhead{Pugh's closing lemma.}
The first result was shown by Pugh~\cite{pugh1,pugh2,pugh-robinson,arnaud1}.
It allows to create by perturbation some periodic orbit once the dynamics is recurrent.
More precisely, one considers the points $z$ whose orbit is {\em non-wandering}:
for any neighborhood $U$ of $z$, there is a forward iterate
$f^n(U)$ of $U$ (with $n\geq 1$) which intersects $U$.

The local perturbation result is the following (see also figure~\ref{f.closing}):
\begin{figure}[ht]
\begin{center}
\input{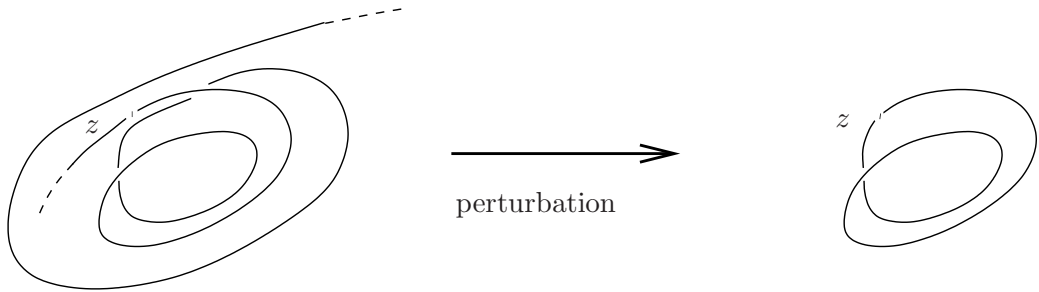}
\end{center}
\caption{Perturbation given by the closing lemma. \label{f.closing}}
\end{figure}
\begin{theo}[Closing lemma, Pugh]
Let $f$ be a $C^1$-diffeomorphism in $\diff_v^1$ and $z\in M$ a non-wandering point.
Then, there exists a $C^1$-small perturbation $g\in \diff_v^1$ of $f$ such that $z$ is a periodic orbit
of $f$.
\end{theo}

\subsubsectionruninhead{Hayashi's connecting lemma.}
We have seen that the closing lemma allows to connect an orbit to itself. About 30 years later, Hayashi~\cite{hayashi, wen-xia, arnaud2} proved a second local perturbation lemma and
showed how to connect an orbit to another one (see figure~\ref{f.connecting}).
\begin{figure}[ht]
\begin{center}
\input{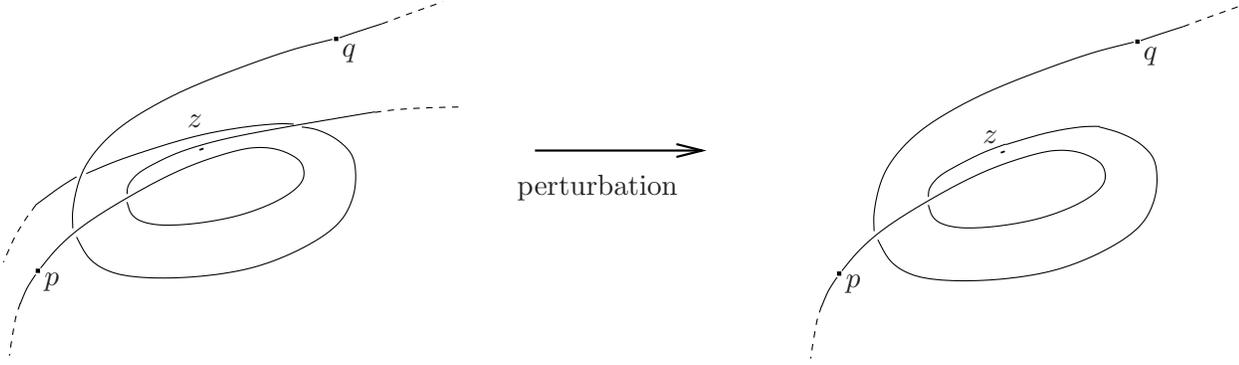}
\end{center}
\caption{Perturbation given by the connecting lemma. \label{f.connecting}}
\end{figure}

\begin{theo}[Connecting lemma, Hayashi]\label{t.connecting}
Let $f$ be a $C^1$ diffeomorphism in $\diff_v^1$ and $p,q,z\in M$ three points such that:
\begin{itemize}
\item both the accumulation sets of the forward orbit of $p$ and of the backward orbit of $q$
contain the point $z$;
\item the point $z$ is not periodic.
\end{itemize}
Then, there exists a $C^1$-small perturbation $g\in\diff^1_v$ of $f$ and an integer $n\geq 1$ such that
$g^n(p)=q$.
\end{theo}
The second assumption in the connecting lemma is purely technical (and maybe not essential).
Some stronger versions of this result are given below.

\begin{rema}
In the closing and connecting lemmas, the perturbations are local: there exists an integer $N\geq 1$ such that
the support of the perturbation is contained in an arbitrarily small neighborhood of the segment of orbit
$\{z,f(z),\dots,f^{N-1}(z)\}$.
\end{rema}

\subsubsectionruninhead{The connecting lemma for pseudo-orbits.}
One can restate Hayashi's connecting lemma in the following form. It is possible to connect $p$ to $q$
by perturbation provided that these points are almost on the same orbit: at some place (close to $z$)
one allows a small jump between the forward orbit of $p$ and the backward orbit of $q$.
The connecting lemma for pseudo-orbits, proved in~\cite{BC,arnaud-bonatti-crovisier}
allows to deal with any number of jumps
(see figure~\ref{f.pseudo-connecting}).
\begin{figure}[ht]
\begin{center}
\input{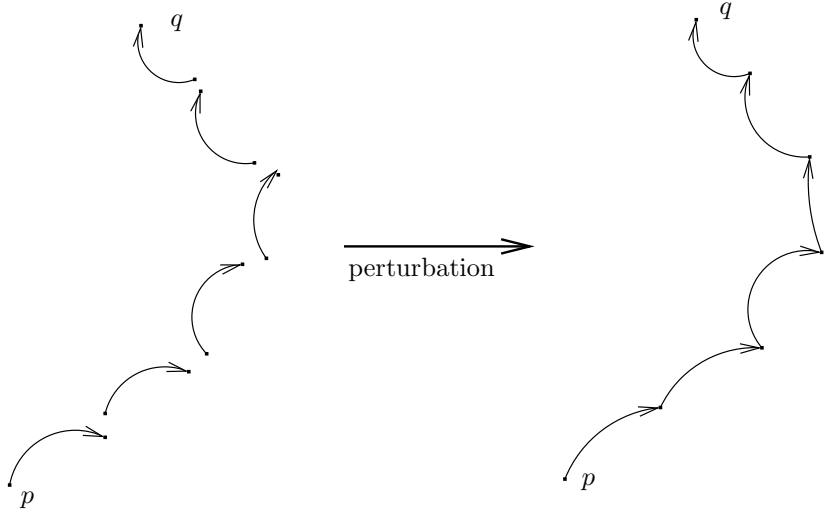}
\end{center}
\caption{Perturbation given by the connecting lemma for pseudo-orbits. \label{f.pseudo-connecting}}
\end{figure}

For any $\varepsilon>0$, we say that a sequence $(z_0,\dots,z_\ell)$ is a $\varepsilon$-pseudo-orbit of $f$
if for any $k\in \{0,\dots,\ell-1\}$, we have $\dd(f(z_k),z_{k+1})<\varepsilon$. In other terms, this sequence
is an orbit with small errors, bounded by $\varepsilon$, at each iterations.

\begin{theo}[Connecting lemma for pseudo-orbits, Bonatti-Crovisier, Arnaud-B-C]
\label{t.pseudo-connecting}
Let $f$ be a $C^1$-diffeomorphism in $\diff^1_v$ such that
for each $\tau\geq 1$, the set periodic points with period $\tau$ of $f$ is finite.
Let $p,q\in M$ be two points such that
for each $\varepsilon>0$, there exists a $\varepsilon$-pseudo-orbit
$(p=z_0,z_1,\dots,z_{\ell-1},z_\ell=q)$ which joints $p$ to $q$.

Then, there exists a $C^1$-small perturbation $g\in\diff^1_v$ of $f$ and an integer $n\geq 1$ such that
$g^n(p)=q$.
\end{theo}
By corollary~\ref{c.periodic-finite}, the first assumption of theorem~\ref{t.pseudo-connecting}
is generic in $\diff^1_v$ (here again, this technical assumption is probably not essential). We will see below that
the second one is always satisfied (we assumed that $M$ is connected).

\subsection{$C^1$-generic properties}
We now give some consequences of these perturbation lemmas.

\subsubsectionruninhead{Density of periodic points.}\label{ss.density}
The first consequence shows the important role played by the periodic orbits in $C^1$-dynamics.
For any diffeomorphism $f$ and any non-empty open set $U$, using the fact that $f$ preserves a smooth
finite measure, we get that a forward iterate $f^n(U)$ of $U$ intersects $U$. This shows that
any point in $M$ is non-wandering.

Now, Pugh's closing lemma implies:
\begin{theo}[Pugh]\label{t.coro-closing}
For any generic diffeomorphism $f\in \diff_v^1$, the periodic points of $f$ are dense in $M$.
\end{theo}
We recall that according to Robinson's theorem~\ref{t.periodic-generic},
each periodic orbit is either elliptic or hyperbolic.

The proof follows from the closing lemma by a classical argument.
We give it as an example of Baire theory.
\begin{proof}[Sketch of the proof of theorem~\ref{t.coro-closing}]
Let $(U_n)$ be a countable basis of neighborhoods in $M$. We fix $n$ and we have to show that
any generic diffeomorphism has a periodic point in $U_n$.

By the implicit function theorem, the diffeomorphisms which have a hyperbolic or elliptic periodic point in $U_n$
form an open set $\cU_n$. It now sufficient to show that for each $n$, the set $\cU_n$ is dense:
the countable intersection $\bigcap\cU_n$ then will be a dense $G_\delta$ set, hence generic.

Let us consider any diffeomorphism $f_0$. Since any point in $U_n$ is non-wandering, by a small perturbation, one
can create in $U_n$ a periodic point $p$ for a diffeomorphism $f_1$ close to $f_0$.
For a new perturbation $f_2$
(given by a transversality argument extracted from the proof of theorem~\ref{t.periodic-generic})
the point $p$ will be elliptic or hyperbolic, as required.
\end{proof}

\subsubsectionruninhead{Study of the elliptic orbits.}
According to the last result, it should be interesting to study separately the two kinds of periodic orbits
that appear in generic dynamics.
The elliptic behavior was studied intensively but mostly in higher differentiability.
The following result is not specific to the $C^1$-topology.

\begin{theo}[Zehnder\cite{zehnder1}]\label{t.density-saddle}
Let $f$ be a generic diffeomorphism in $\diff_v^1$.
Any elliptic periodic orbit is accumulated by hyperbolic saddles. Hence, the hyperbolic periodic points
are dense in $M$.
\end{theo}

We give a different proof from Zehnder's original one (which used approximations by smooth diffeomorphisms).
We need a useful perturbation lemma by Franks which allows to change the tangent map along the periodic orbits.
\begin{theo}[Franks lemma, \cite{franks2,bonatti-diaz-pujals}]\label{t.franks}
Let $q$, $f(q)$,\dots, $f^\tau(q)=q$ be a periodic orbit of a conservative diffeomorphism $f$.
Then, any $C^0$-perturbation of the sequence of linear tangent maps $(\D_q f, \D_{f(q)}f$,\dots, $\D_{f^{\tau-1}(q)}f)$,
that preserves the volume forms $v$, can be realized as the sequence of tangent maps
$(\D_qg, \D_{g(q)}g$,\dots, $\D_{g^{\tau-1}(q)}g)$ associated to a conservative diffeomorphism $g$
that is $C^1$-close to $f$ and preserves the orbit $\{q,f(q), \dots,f^{\tau-1}(q)\}$.
\end{theo}

\begin{proof}[Sketch of the proof of theorem~\ref{t.density-saddle}]
Let us consider some elliptic periodic point $p$. We explain how to create by perturbation a hyperbolic
saddle $q$ close to $p$. Theorem~\ref{t.density-saddle} then follows from a standard Baire argument.

We first recall that the derivative $\D_pf^\tau$ along the orbit of $p$ is conjugate to a rotation.
By a transversality argument, one may perturb the dynamics so that the angle of this rotation is irrational.
The dynamics in a neighborhood of the orbit of $p$ looks like the dynamics of an irrational rotation.
Hence, by a small perturbation one creates in a neighborhood of the orbit of $p$ a second periodic point $q$
whose period $\tau'$ is an arbitrarily large multiple $k\tau$ of $\tau$: this argument uses Pugh's closing lemma for
dynamics $C^1$-close to the rotation.

Note that the derivative along the orbit of $q$ is close to the rotation $(\D_pf^\tau)^k$.
Since the period $k\tau$ of $q$ can be chosen arbitrarily large,
one can perturb the derivative at points $q$, $f^\tau(q),\dots,f^{k-1}\tau(q)$ (by compositing by small rotations)
in order to obtain a sequence of linear maps whose product is equal to the identity.
By using Franks lemma (theorem~\ref{t.franks}), one gets a $C^1$-small perturbation of $f$
whose derivative along the orbit of $q$ is the identity.
By a new arbitrarily $C^1$-small perturbation, the orbit of $q$ can be made of saddle type.
\end{proof}

\subsubsectionruninhead{Study of the hyperbolic saddles.}
We have seen that the hyperbolic periodic orbits are dense in $M$.
The next result shows the existence of {\em homoclinic intersections}, i.e. of transverse intersections between
the stable and unstable manifolds of these periodic orbits.
This is important since a transverse intersection between the invariant manifolds of a saddle
periodic orbit implies the existence of non-trivial
{\em hyperbolic sets}: a compact and invariant set is {\em hyperbolic}
if the tangent bundle over $K$ splits as two one-dimensional bundles, one is uniformly contracted and
the other one is uniformly expanded. The periodic saddle are the simplest examples of hyperbolic
sets but Smale has shown that the homoclinic intersections imply the existence of larger hyperbolic sets,
that are Cantor sets, and called {\em horseshoes}. The dynamics on horseshoes is very rich, but has been
well described.

\begin{theo}[Takens~\cite{takens}]\label{t.horseshoe}
For any generic diffeomorphism $f$ and any periodic point $p$,
the transverse intersection points between the invariant manifolds $W^s(p)$ and $W^u(p)$ of $p$
are dense in $W^s(p)$ and $W^u(p)$.
\end{theo}
The proof now is easy from the connecting lemma.
\begin{figure}[ht]
\begin{center}
\input{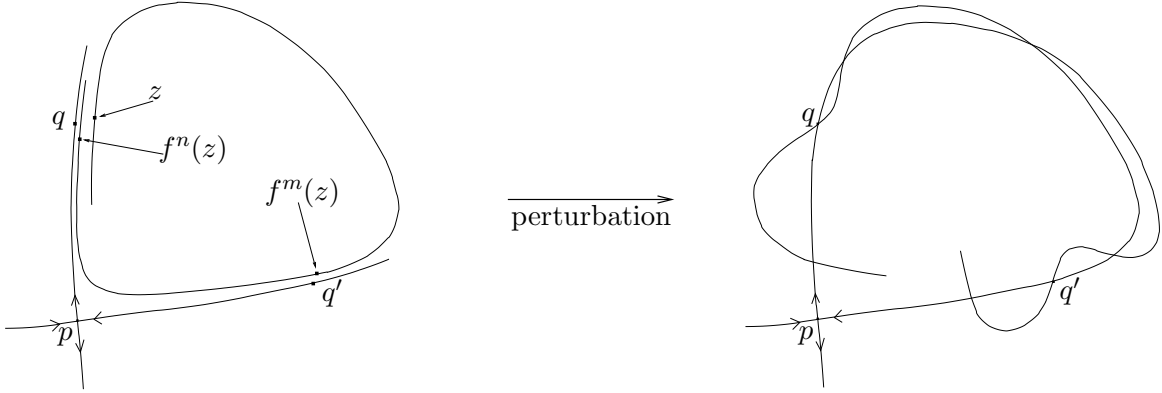}
\end{center}
\caption{Creation of a homoclinic point by perturbation. \label{f.homoclinic}}
\end{figure}

\begin{proof}[Sketch of the proof of theorem~\ref{t.horseshoe}]
Let $p$ be a saddle of some diffeomorphism $f$.
We show how to build one intersection between the stable and unstable manifolds.

Let us choose some point $q$ on the unstable manifold of $p$. Since $q$ is non-wandering, it is accumulated
by a sequence $(z_k)$ and by a sequence of forward iterates $(f^{n_k}(z_k))$ of the points $z_k$.
We note that each sequence $(z_k,f(z_k),\dots,f^{n_k}(z_k))$ should come close to the stable manifold of $p$
before visiting $q$. Hence, there is a sequence $(f^{m_k}(z_k))$ with $1<m_k<n_k$ which converges to
a point $q'$ on the stable manifold of $p$.

Now, the connecting lemma\footnote{In fact, we use here a variation on the connecting lemma
given later in the text. This variation asserts than one chooses first the place where the perturbation is realized
and then the orbits that should be connected. This will be stated precisely at section~\ref{ss.conclusion-connecting}.}
(used twice) allows to connect the unstable manifold at $q$ and the stable manifold at $q'$ to a
same segment of orbit $(z_k,f(z_k),\dots,f^{m_k}(z_k))$ (see figure~\ref{f.homoclinic}).
One gets an intersection at $q$ between the stable and unstable manifolds of $p$.
By a new perturbation, this intersection becomes transverse.
\end{proof}

\subsubsectionruninhead{Global dynamics.}
Up to here, we proved that the periodic points, and saddles, are dense in $M$.
One could imagine however that $M$ has some invariant domains so that each periodic orbit stay in a small region
of $M$. We will see that this picture is wrong generically.

We state before a property of conservative diffeomorphism which justifies the role of pseudo-orbits:
\begin{prop}\label{p.chainrecurrence}
For any diffeomorphism $f\in\diff_v^1$ and any $\varepsilon>0$, any two points $p,q\in M$
can always be jointed by an $\varepsilon$-pseudo-orbit.
\end{prop}
\begin{proof}
We recall Poincar\'e recurrence theorem: almost any point $x\in M$ is {\em recurrent}: the accumulation set of the
forward orbit of $x$ contains $x$. In particular, the recurrent points are dense in $M$.\\
Let us fix some $\varepsilon>0$ and denote by $X\subset M$ be the set of points that can be attained from
$p$ by a $\varepsilon$-pseudo-orbit. We choose $x\in X$. There exists a $\varepsilon$-pseudo-orbit
from $p$ to $x$. It is possible to change a little bit the point $x$ so that $x$ is a recurrent point.

Hence, one can consider a return $f^\ell(x)$ of the orbit of $x$ which is close to $x$. This shows that
any point $y$ in the ball $\B(x,\varepsilon/2)$ can be jointed from $x$ by a $\varepsilon$-pseudo-orbit.
By concatening the pseudo-orbit from $p$ to $x$ with the pseudo-orbit from $x$ to $y$, one deduces that
$y$ may be attained from $p$ by a $\varepsilon$-pseudo-orbit.
Thus, $\B(x,\varepsilon/2)$ is contained in $X$, for any point $x\in X$.
Since $M$ is connected, one deduces that $X=M$ and in particular $q\in X$.
\end{proof}

As a consequence of the connecting lemma for pseudo-orbits, one gets the following global result:
\begin{theo}[Bonatti-Crovisier]\label{t.transitivite}
Any generic diffeomorphism $f\in \diff^1_v$ is transitive: there is a G$_\delta$ and dense subset
$\cG\subset M$ such that the forward orbits of any $x\in \cG$ is dense in $M$.
\end{theo}
\begin{proof}
Let $(U_n)$ be a countable basis of neighborhoods of $M$.
Using proposition~\ref{p.chainrecurrence}
and the connecting lemma for pseudo-orbits, we get that for each open set $U_n$, the
set $\cG_n\subset M$ of points whose forward orbit meets $U_n$ is dense in $M$. It is also open.
Hence, $\cG=\bigcap \cG_n$ is a dense G$_\delta$ subset of $M$ whose points have a forward orbit dense in $M$.
\end{proof}
An other consequence of theorem~\ref{t.pseudo-connecting}, is that for a generic diffeomorphism $f$ and any pair of periodic
saddles $p$ and $q$, the invariant manifolds $W^s(f^i(p))$ and $W^u(f^j(q))$ of some iterates of $p$ and $q$
intersect transversally. One says that $f$ possesses a unique {\em homoclinic class}.

This property implies the existence of many other periodic points. For example, there exists a periodic orbit
(whose period is very large) which shadows the orbit of $p$ during a large number of iterates and then
shadows the orbit of $q$. Since the periodic saddles are dense in $M$
(by theorem~\ref{t.density-saddle}), we get the announced property:
\begin{coro}
For any generic diffeomorphism in $\diff^1_v$ and any $\varepsilon$, there exists a periodic orbit
which is $\varepsilon$-dense in $M$ (i.e. which meets any ball of radius $\varepsilon$ in $M$).
\end{coro}

\subsubsectionruninhead{Repartition of periodic orbits.}
It is possible to improve the density theorem~\ref{t.coro-closing}
by describing the repartition of the periodic orbits.
From the measure theory, this is the ergodic closing lemma due to Ma\~n\'e~\cite{mane}.
\begin{theo}[Ergodic closing lemma, Ma\~n\'e]
Let $f$ be a generic diffeomorphism in $\diff_v^1$. Then, any invariant probability measure\footnote{We do not assume
here that the measure $\mu$ is ergodic. Hence, a consequence of this result is that in the convex and compact set
$\cM$ of invariant measures, the ergodic measures (these are the extremal points of $\cM$) are dense.} $\mu$
of $f$ is the weak limit of periodic measures $(\mu_n)$ whose supports converge towards the support of $\mu$.
\end{theo}

From the topological point of view, one may wonder what are the regions of $M$ that are shadowed
by a single periodic orbit. This requires another global perturbation
lemma that we won't detail here (see~\cite{crovisier}).

We say that an invariant compact set $K$ is {\em chain transitive} if for any points $p,q\in K$
and any $\varepsilon>0$, there is a $\varepsilon$-pseudo-orbit contained in $K$ that joints $p$ to $q$.
The result is the following:

\begin{theo}[Crovisier]
For any generic diffeomorphism in $\diff^1_v$ and any invariant compact set $K$,
the set $K$ is the Hausdorff limit of a sequence of periodic orbits\footnote{This means
that for any $\delta>0$, there exists a periodic orbit contained in the $\delta$-neighborhood of $K$
and which crosses all the $\delta$-balls centered at points of $K$: at scale $\delta$, the periodic orbit
$\cO$ and the set $K$ can not be distinguished.}
if and only if it is chain-transitive.
\end{theo}

As a consequence, we get a weak shadowing property: pseudo-orbits whose jumps are very small ``can not be distinguished"
from ``true orbits".

\begin{coro}
For any generic diffeomorphism $f\in\diff^1_v$, for any $\delta>0$, there exists $\varepsilon>0$ such that any $\varepsilon$-pseudo-orbit
$PO=\{z_0,\dots,z_m\}$ of $f$ is $\delta$-close to a genuine segment of orbit
$O=\{z,f(z),\dots,f^n(z)\}$ of $f$: this means that $PO$ is contained in the $\delta$-neighborhood of $O$ and that
$O$ is contained in the $\delta$-neighborhood of $PO$.
\end{coro}

\subsection{Hyperbolic versus elliptic dynamics}
Whereas the hyperbolic saddles of a generic diffeomorphism always exists and are dense in $M$,
the same is not true for the elliptic points.

\subsubsectionruninhead{Hyperbolic behavior: Anosov diffeomorphisms.}
An important example of conservative surface dynamics
are the {\em Anosov diffeomorphisms}: these are the diffeomorphisms such that the whole manifold $M$ is
a hyperbolic set.

For example, any linear automorphism in $SL(2,\ZZ)$ acts on the two-torus $\TT^2=\RR^2/\ZZ^2$
and preserves the canonical Haar measure. When the modulus of the eigenvalues is different from $1$,
one gets an Anosov diffeomorphism.

The general Anosov diffeomorphisms are important because they form an open subset of $\diff_v^k$.
Their dynamics are well understood since Franks has shown~\cite{franks1} that any Anosov surface diffeomorphism
is conjugate by a homeomorphism to a linear Anosov automorphism of $\TT^2$.

\subsubsectionruninhead{The dynamics far away from the Anosov diffeomorphisms.}
It is clear that Anosov diffeomorphisms don't have any elliptic periodic orbit.
For the other diffeomorphisms, Newhouse showed~\cite{newhouse} that the situation is completely different.
\begin{theo}[Newhouse]
For any generic diffeomorphism in $\diff^1_v$, which is not Anosov, the elliptic periodic points
are dense in $M$.
\end{theo}
In particular, if $M$ is not the torus, we get the same conclusion for any generic diffeomorphism in $\diff^1_v$.
Combining the technics of~\cite{bonatti-diaz-pujals} and \cite{BC}, one can improve this result and show that
for any generic diffeomorphism which is not Anosov, for any $\varepsilon>0$, there exists an elliptic periodic orbit
which is $\varepsilon$-dense in $M$.

We note also that if $f$ is an Anosov diffeomorphism, any diffeomorphism $g$ that is $C^1$-close to $f$ is transitive
($f$ is said {\em robustly transitive}).
This property characterizes the Anosov diffeomorphisms.

\begin{prop}
For any diffeomorphism $f\in\diff^1_v$ which is not Anosov, there is a $C^1$-small perturbation $g$
which is not transitive.
\end{prop}
\begin{proof}
Using remark~\ref{r.smooth}, one can approach $f$ by a smooth diffeomorphism $\bar f$ which has an elliptic periodic orbit.
For a new perturbation $g$, the assumptions of KAM theorem~\ref{t.KAM} are satisfied and $g$ is not transitive.
\end{proof}

\subsubsectionruninhead{The Lyapunov exponents.}
The dichotomy hyperbolic/elliptic can be also detected from the Lyapunov exponents of the diffeomorphism:
Oseledets theorem asserts that for any diffeomorphism $f\in\diff^1_v$ and at $v$-almost any point,
the {\em upper Lyapunov exponent} exists:
$$\lambda^+(f,x)=\lim_{n\rightarrow +\infty} \frac{1}{n}\log\|\D_x f^n\|.$$
This quantity, which is always non-negative, describes how the infinitesimal dynamics along the orbits of $x$
is stretched.

Bochi proved~\cite{bochi} the following property.
\begin{theo}[Bochi]
For any generic diffeomorphism $f\in\diff^1_v$ two cases can occur:
\begin{itemize}
\item either $f$ is Anosov, (and $\lambda^+(f,x)$ is strictly positive at any point $x$ where it is defined);
\item or $\lambda^+(f,x)=0$ for $v$-almost every point $x$.
\end{itemize}
\end{theo}
The proof is very interesting since it combines perturbation technics of the derivative along some
segments of orbits and a control of the measure of the points which exhibit the required behavior.

\subsection{Questions}
We conclude this survey with some open problems.
\begin{enumerate}
\item {\em Is any generic diffeomorphism topologically mixing?}\\
A diffeomorphism $f$ is {\em topologically mixing} if
for any non-empty open sets $U$ and $V$, there is an integer $n_0\geq 1$
such that $f^n(U)$ intersects $V$ once $n\geq n_0$. It is a stronger property than the transitivity.
In fact, theorem~\ref{t.pseudo-connecting} implies that
for a generic diffeomorphism, the manifold splits into a finite number of pieces
$\Lambda$, $f(\Lambda)$,\dots, $f^\tau(\Lambda)=\Lambda$ that are cyclically permuted by the dynamics.
The return map $f^\tau$ on $\Lambda$ is topologically mixing. The problem here is to decide whether
there is only one piece ($\tau=1$).
\item {\em Is any generic diffeomorphism ergodic?}\\
A weaker problem would be: is there a full measure set of point whose orbits are dense?
\item {\em Is there a closing lemma in higher differentiability?}\\
About perturbation lemmas on surfaces and in higher differentiability, there exist some partial results that adopt a topological approach,
see~\cite{robinson2,mather,pixton,oliveira,franks-lecalvez}.
\end{enumerate}
\vspace{1cm}

\newpage
\section{Local perturbations in $C^1$-dynamics}\label{s.local}
In this part we explain the main ideas for closing or connecting orbits by perturbation.
Most of the technical details that are skipped here can be read in~\cite{pugh-robinson,arnaud1,wen-xia0,BC}.

\subsection{The elementary perturbation lemma}\label{s.elementary}
\subsubsectionruninhead{Statement.}
The basic result that allows perturbations is the elementary perturbation lemma
introduced at section~\ref{ss.elementary}. Since the perturbation is local, we state it in $\RR^2$
without loosing in generality (see figure~\ref{f.elementary}).
\begin{prop}[Elementary perturbation lemma]\label{p.elementary}
For any neighborhood $\cU$ of $\id$ in the space of conservative $C^1$-diffeomorphisms of $\RR^2$,
there exists $\eta>0$ with the following property:\\
If $x,y\in \RR^2$ are close enough to each other there is $\varphi\in \cU$
which sends $x$ onto $y$ and coincides with $\id$ outside the ball $\B\left(\frac{x+y}2,\frac{1+\eta}2\dd(x,y)\right)$.
\end{prop}
One gets the perturbation $g$ of $f$ announced at section~\ref{ss.elementary}, as a composition $\varphi\circ f$.
\begin{figure}[ht]
\begin{center}
\input{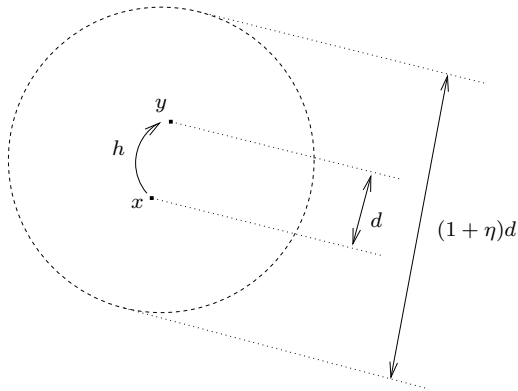}
\end{center}
\caption{An elementary perturbation. \label{f.elementary}}
\end{figure}

\subsubsectionruninhead{Why this result is not sufficient to our purpose?}
One can hope that this result will be sufficient to get the closing lemma:
if one considers a segment of orbit $(p,f(p),\dots,f^n(p))$ and if $p$ and $f^n(p)$ are
close enough one may want to introduce a perturbation $\varphi$ given by proposition~\ref{p.elementary}
which sends $f^n(p)$ on $p$. However the point $p$ will not be periodic for the diffeomorphism
$g=\varphi\circ f$ in general. The reason is that the segment of orbit $(p,f(p),\dots,f^n(p))$
could have many intermediate returns in the support of the perturbation $\varphi$ and could be broken
by the perturbation (figure~\ref{f.break}).
\begin{figure}[ht]
\begin{center}
\input{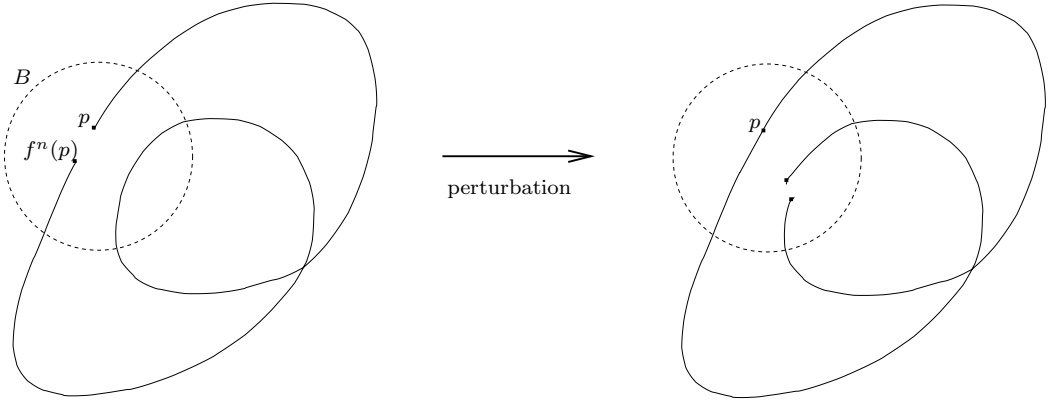}
\end{center}
\caption{Orbit broken by the perturbation. \label{f.break}}
\end{figure}
 
\subsubsectionruninhead{Proof of the closing lemma when $\eta=\frac1 2$.}\label{ss.proof-closing}
When $\eta\leq \frac 1 2$ (but in this case the $C^1$-size of the perturbation is not so small)
one can deduce the closing lemma directly from the elementary perturbation lemma.
The main difficulty is to select two close iterates that could be jointed to each other by a perturbation
(figure~\ref{f.selection}).
\begin{figure}[ht]
\begin{center}
\input{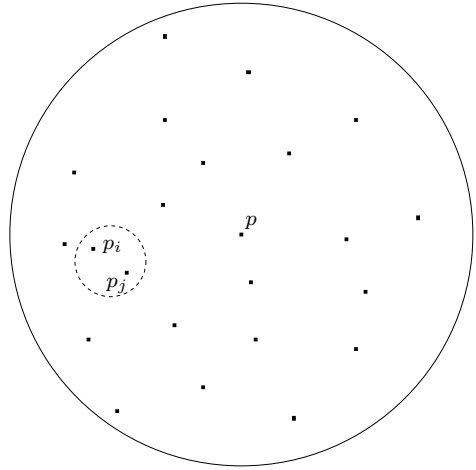}
\end{center}
\caption{Selection of returns in $B$. \label{f.selection}}
\end{figure}

Let us consider a point $p$ whose orbit is recurrent\footnote{The assumption in the closing lemma is
that $p$ is non-wandering. In order to simplify the proof,
our assumption here is a little bit stronger but the argument is the same.}:
the forward orbit of $p$ has some returns arbitrarily close to $p$.
We fix also a small ball $B=\B(p,r)$ centered at $p$, that will contain the support of the
perturbation. We choose a time $n\geq 1$ such that $f^n(p)$ and $p$ are close enough.

Let us denote by $p_0=p, p_1,\dots,p_{s-1}, p_s=f^n(p)$ the points of
$\{p,f(p),\dots,f^n(p)\}$ that belong to the interior of $B$.
Among the pairs $(p_i,p_j)$ with $i<j$, we choose one which minimizes the quantity
$$D_{i,j}=\frac{\dd (p_i,p_{j})}{\dd\left(\frac{p_i+p_{j}}2, \partial B\right)}.$$
This plays the role of a ``hyperbolic distance in the ball $B$" between the points $p_i$.
Note that since the distance between $x$ and $f^n(x)$ has been chosen arbitrarily small
in comparison to the radius $r$ of $B$, the minimum of the $D_{i,j}$ can be assumed arbitrarily small.

An easy estimate gives the following claim.
\begin{claim}
If $D_{i_0,j_0}$ is small enough and realizes the minimum over all the $D_{i,j}$,
the ball $\B\left(\frac{p_{i_0}+p_{j_0}} 2, \frac {1+\eta} 2 \dd (p_{i_0},p_{j_0})\right)$
is contained in $B$ and does not intersect any other point $p_0,\dots,p_s$.
\end{claim}
In particular, one now can realize the perturbation $\varphi$ which sends $p_{j_0}$ onto $p_{i_0}$.
The segment of orbit $(p_{i_0},f(p_{i_0}),\dots, f^{-1}(p_{j_0}))$ is not perturbed by the perturbation and
the point $p_{i_0}$ is now periodic.
Hence, we get a periodic point close to the initial point $p$. A small perturbation
(one conjugates by a translation) can move the periodic point $p_{i_0}$ onto $p$, as required.

\subsubsectionruninhead{Composition of perturbations.}\label{ss.composition}
In the general case, a unique elementary perturbation $\varphi$ is not sufficient and we will
realize several perturbations at different places. We will now see that this is allowed, since the size of the perturbation
one obtains do not increase with the number of independent elementary perturbations that have been realized.

If $\cU$ is a neighborhood of a diffeomorphism $f$, the {\em support} of a perturbation $g\in \cU$ is the
set of points $x$ where $f(x)$ and $g(x)$ differ. If $g_1=\varphi_1\circ f$ and $g_2=\varphi_2\circ f$
are two perturbations with distinct supports, the {\em composition of the perturbations} is the diffeomorphism
$g=\varphi_1\circ \varphi_2\circ f=\varphi_2\circ \varphi_1\circ f$.
By shrinking the neighborhood $\cU$ if necessary, one can choose it with the {\em composition property}:
$$ g_1, g_2\in \cU \; \Rightarrow \; g\in \cU.$$
Hence, one may compose an arbitrary number of elementary perturbations without leaving the neighborhood $\cU$ of $f$.

\subsection{The closing lemma}
In order to prove the closing lemma, we would like to realize a small perturbation given by proposition~\ref{p.elementary}
with a constant $\eta\leq\frac 1 2$.
As we explained in section~\ref{ss.elementary} this is not possible.
The idea of Pugh was to spread the perturbation in the time and to obtain it as a composition of several elementary
perturbation. This idea allows roughly to divide the (a priori large) constant $\eta$, given by the size of the
perturbations that are allowed, by the time we consider to spread the perturbation.

\subsubsectionruninhead{Pugh's perturbation lemma.}
\begin{theo}\label{t.pugh}
Let $f\in\diff_v^1$ be a diffeomorphism and let us consider a point $z$ which is not periodic.
Then, for any neighborhood $\cU$ of $f$ in $\diff_v^1$,
there exists an integer $N$ and a Riemannian metric $\dd'$ which have the following property:

If $x,y$ are two points contained in a small ball $S=\B_{\dd'}(z,\delta)$ centered at $z$, for the metric $d'$,
there is a perturbation $g\in \cU$ of $f$ which sends $x$ onto $f^N(y)$ by $g^N$.

The support of the perturbation $g$ is contained in the union $\hat S\cup f(\hat S)\cup\dots\cup f^{N-1}(\hat S)$
of the ball $\widehat S = \B_{\dd'}(z,\frac{3}{2}\delta)$ with its $N-1$ first iterates.
\end{theo}
\begin{figure}[ht]
\begin{center}
\input{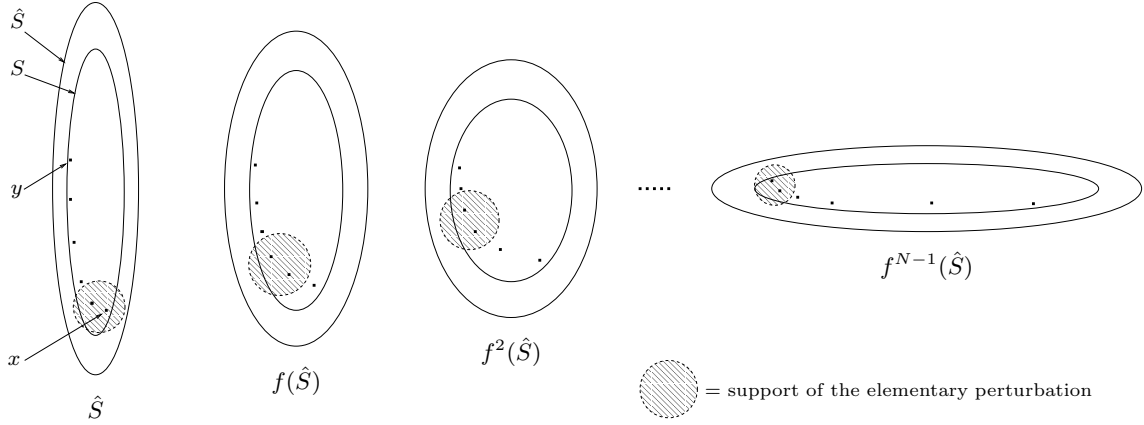}
\end{center}
\caption{Perturbations in theorem~\ref{t.pugh}. \label{f.pugh-lemma}}
\end{figure}
Note that since $z$ is not periodic, the ball $\widehat S$ and its $N-1$ first iterates are disjoint.
With this result we ``recover a constant $\eta=\frac 1 2$": the perturbation around $x$ and $y$
occurs in the ball of diameter less than $1+\frac 1 2$ times the diameter of the ball $S$ which contains $x$ and $y$.
The price to pay is that we also perturb in the $N-1$ iterates of the ball
$\widehat S$ (figure~\ref{f.pugh-lemma}).
However, the argument given at section~\ref{ss.proof-closing} remains and theorem~\ref{t.pugh} implies the closing lemma.

\begin{rema}\label{r.small-support}
Note that all the sets $\widehat S$, $f(\widehat S)$,\dots, $f^{N-1}(\widehat S)$ have roughly the shape of ellipsis.
An important improvement (used in the connecting lemma below) is that in each ellipsis $f^k(\widehat S)$ with
$k\in \{0,,\dots, N-1\}$, the support of the perturbation $g$ is contained in a ball which is small
with respect to the smaller axis of $f^k(\widehat S)$.
\end{rema}

\subsubsectionruninhead{Proof of theorem~\ref{t.pugh} when $f$ is conformal.}
In order to explain how to spread the perturbation in the time and get theorem~\ref{t.pugh},
we first consider
the case where $f$ is ``conformal" (i.e. the image of a small Euclidean ball is roughly a small euclidian ball).
One considers the (large) constant $\eta$ given by the elementary perturbation lemma, a small ball
$S=\B(z,\delta)$ around $z$ (for the standard metric) and two points $x,y\in S$.

We choose $N=4(1+\eta)$ and we divide the segment between $x$ and $y$ by a sequence $(\zeta_0=x,\dots,\zeta_N=y)$
of points at distance $\dd(x,y)/N$ from each other.

For each $i\in\{0,\dots,N-1\}$, the image of the ball $\widehat S=\B(z,\frac 3 2 \delta)$ by $f^i$
is roughly a ball, by assumption.
The two points $f^i(\zeta_i),f^i(\zeta_{i+1})$ are contained in this ball. Moreover, by our assumption,
the relative distance
between these two points in comparison to their distance to the boundary of $f^i(\widehat S)$ is small:
it is close to the relative distance between $\zeta_i$ and $\zeta_{i+1}$
in comparison to their distance to the boundary of $\hat S$.
So that the elementary perturbation lemma gives a
perturbation $g_i=\varphi_i\circ f$ in $\cU$ with support in $f^i(\widehat S)$ (figure~\ref{f.conformal}) such that
$g_i(f^i(\zeta_i))=f^{i+1}(\zeta_{i+1})$.
\begin{figure}[ht]
\begin{center}
\input{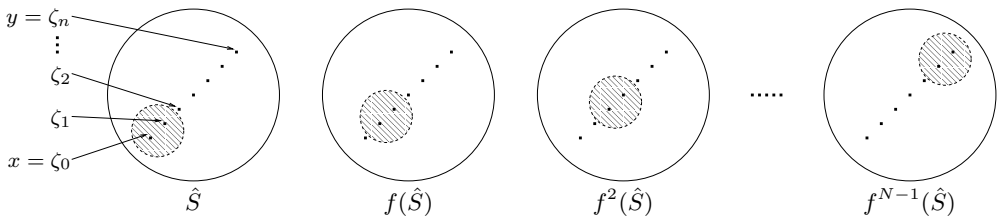}
\end{center}
\caption{Perturbation when $f$ is conformal. \label{f.conformal}}
\end{figure}

All the perturbations $g_i$ have by construction disjoint supports so that the composed perturbation
$g=(\varphi_0\circ \varphi_1\circ\dots\circ \varphi_{N-1})\circ f$ belongs to $\cU$ (by the composition property
of section~\ref{ss.composition}) and satisfies
$g^N(\zeta_0)=f^N(\zeta_N)$, as announced.\\

\subsubsectionruninhead{ Proof of theorem~\ref{t.pugh} when $f$ is not conformal.}
The difficulty when $f$ is not conformal is that the image of a ball is no more a ball (after several iterations,
it could be a ellipsis with a large eccentricity).
Hence, the two points $f^i(\zeta_i),f^i(\zeta_{i+1})$ could be at a small distance from the boundary of the
ellipsis $f^i(\widehat S)$ (in comparison to their relative distance). This is the case in particular if the segment
that joints them follows the direction of the largest axis of the ellipsis. Therefore, the support of perturbation
given by proposition~\ref{p.elementary} which sends $f^i(\zeta_i)$ on $f^i(\zeta_{i+1})$ is no more contained in
$f^i(\widehat S)$.
On the contrary, if one assumes that the segment between these points follows
the direction of the small axis of the ellipsis, the perturbation can be realized inside $f^i(\widehat S)$ as in the
conformal case.
On figure~\ref{f.difficulty}, the right part shows the case where the points are in a bad position so that
the elementary perturbation can not be realized; a good position is pictured on the left part of the figure.
\begin{figure}[ht]
\begin{center}
\input{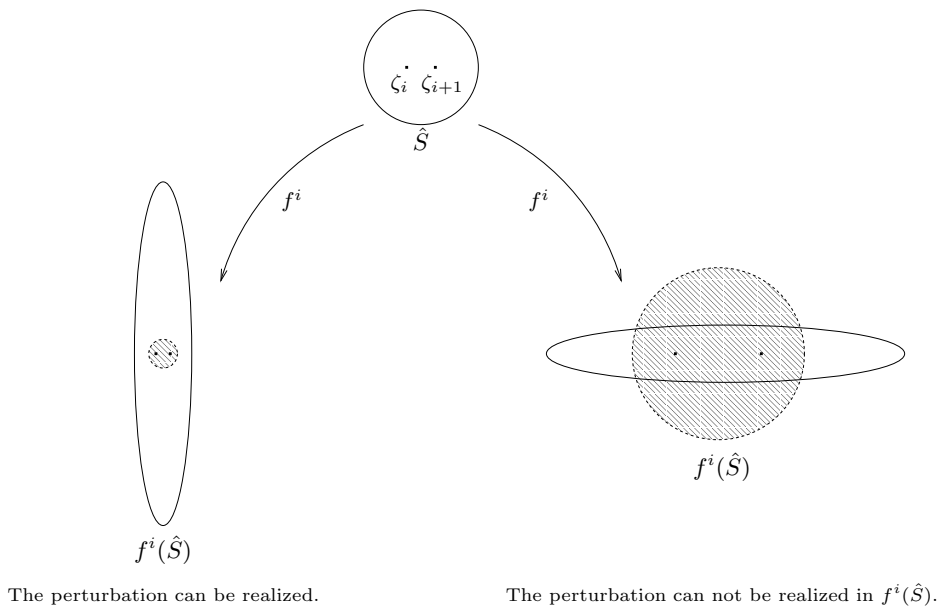}
\end{center}
\caption{Difficulty when $f$ is not conformal. \label{f.difficulty}}
\end{figure}

In order to bypass this problem, we introduce a different metric $\dd'$, so that $S$ and $\widehat S$ are chosen as 
balls for $\dd'$ but as ellipsis for the initial metric. The perturbation is not realized
at each time $i$ (hence the integer $N$ could be larger than in the previous case, depending on the behavior of the
derivatives of $f$ along the orbit of $z$) and the points $(\zeta_i)$ are not
chosen along a segment. What is important is that there are times $0\leq n_0<n_1<\dots<n_i<\dots\leq N$
such that the segment between the points
$f^{n_i}(\zeta_i)$ and $f^{n_i}(\zeta_{i+1})$ follows roughly the direction of the small axis of the ellipsis
$f^{n_i}(\widehat S)$. At these times, one realizes the elementary perturbations.
Figure~\ref{f.pugh-lemma} gives an idea of the way the perturbations are chosen.
We refer the reader to~\cite{wen-xia0} for a good detailed presentation of this proof.

\begin{rema}
One can choose $N$ large enough and the sequences $(\zeta_i)$ and $(n_i)$ carefully so that
at each time $n_i$, the distance between $f^{n_i}(\zeta_i)$ and $f^{n_i}(\zeta_{i+1})$ is small
in comparison to the small axis of $f^{n_i}(S)$. This implies the remark~\ref{r.small-support}.
\end{rema}

\subsection{The connecting lemma}\label{s.connecting}
\subsubsectionruninhead{Why the connecting lemma is more difficult than the closing lemma?}\label{ss.difficulty-connecting}
In the proof of the closing lemma (section~\ref{ss.proof-closing}) we had to select two returns $f^{i_0}(p)$ and $f^{j_0}(p)$ in a ball $B$,
with the property that they should be far enough from the other intermediate iterates.
When one tries to connect one orbit $\{f^n(p)\}$ to another one $\{f^{-m}(q)\}$,
one may also selects two points in $B$ among the returns $\{f^{n_\ell}(p)\}\cup \{f^{-m_k}(q)\}$ of the two orbits.
However if the two selected returns that we get belong both to the first or both to the second orbit, a perturbation will
produce a periodic orbit that crosses $B$ but not an orbit that joints the two points $p$ and $q$.
Hence, one should require that the two selected returns does belong to different orbits.
In general, it is not possible to find a pair of points which has this additional property (figure~\ref{f.connecting-difficulty}).
\begin{figure}[ht]
\begin{center}
\input{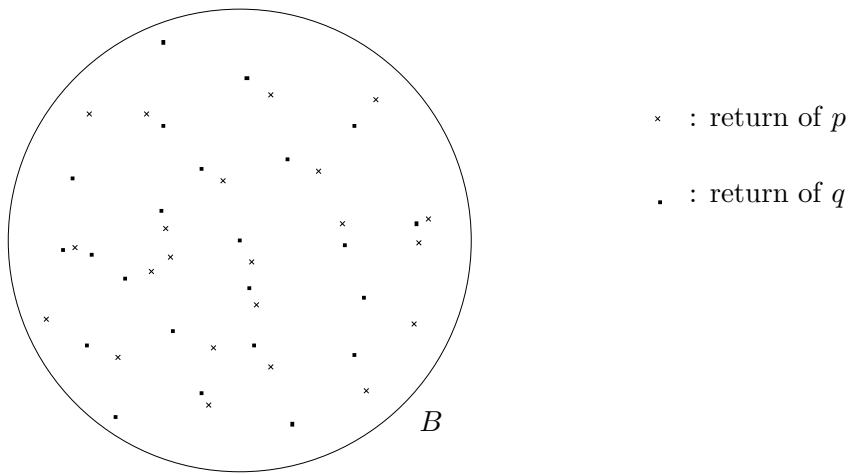}
\end{center}
\caption{A wrong selection of returns for the connecting lemma. \label{f.connecting-difficulty}}
\end{figure}

\subsubsectionruninhead{Hayashi's strategy.}
One idea of Hayashi is to clean up the cloud of returns of the two orbits in $B$ by forgetting some of them.
If two returns of the same orbit are close enough (for example two iterates $f^{n_1}(p)$ and $f^{n_2}(p)$ of $p$),
they prevent us from using the argument of section~\ref{ss.proof-closing} as we explained at~\ref{ss.difficulty-connecting}.
In this case, one will consider that
these two points are the same ($f^{n_1}(p)=f^{n_2}(p)$) and forget the intermediate returns in between.
One should also hope that a small perturbation could move $f^{n_1}(p)$ on $f^{n_2}(p)$ so that the assumption
is fulfilled.
Hence, Hayashi's strategy consists in selecting a large number of pairs of returns and not only one pair as in the closing lemma.
We then realize, for each of these pairs, a perturbation given by Pugh's theorem~\ref{t.pugh} in order to close the orbits.
One difficulty is to guarantee that all these perturbations have disjoint supports.

More precisely, let us consider the first returns $p_0$, $p_1$, \dots, $p_r$ of the forward orbit of $p$ in $B$, ordered
chronologically and the last returns $q_{-s}$, \dots, $q_{-1}$, $q_0$ of the backward orbit of $q$ in $B$, also
ordered chronologically. Recall that the two orbits accumulate on a same point $z$. It is thus possible to choose the ball $B$
centered at $z$ so that one can assume moreover that the last returns
$f^{n(p)}(p)=p_r$ and $f^{-n(q)}(q)=q_{-s}$ are very close to the center of the ball $B$.

In the chronological sequence $(p_0,\dots,p_r,q_{-s},\dots,q_0)$ we extract a subsequence of the form
$(x_0,y_0,x_1,y_1,\dots,x_\ell,y_\ell)$ so that (using Pugh's theorem~\ref{t.pugh}) for each $i$ one can perturb $f$
as a diffeomorphism $g_i=\psi_i\circ f$ which satisfies $g_i^N(x_i)=f^N(y_i)$.
The support of this perturbation is contained in a small ball $\widehat S_i$ contained in $B$, and in the $N-1$ first
iterates of $\hat S_i$.
Moreover, the supports of the different perturbations $(g_i)$ should be pairwise disjoint.

If moreover one has, $x_0=p_0$, $y_\ell=q_0$ and if for each $i$, the point $x_{i+1}$ is the first return of the
orbit of $y_i$ to the ball $B$, then, by composing all these perturbations $g_i$ of $f$,
one gets a diffeomorphism $g=\psi_0\circ \dots\circ \psi_{N-1}\circ f\in\cU$
which sends by some forward iteration the point $p$ on $q$. After perturbation the segment of orbit from $p$ to $q$
is shorter than the initial pseudo-orbit 
$(p,f(p),\dots,f^{n(p)}(p),f^{-n(q)+1}(q),\dots,f^{-1}(q),q)$ (see figure~\ref{f.combinatorics}).
\begin{figure}[ht]
\begin{center}
\input{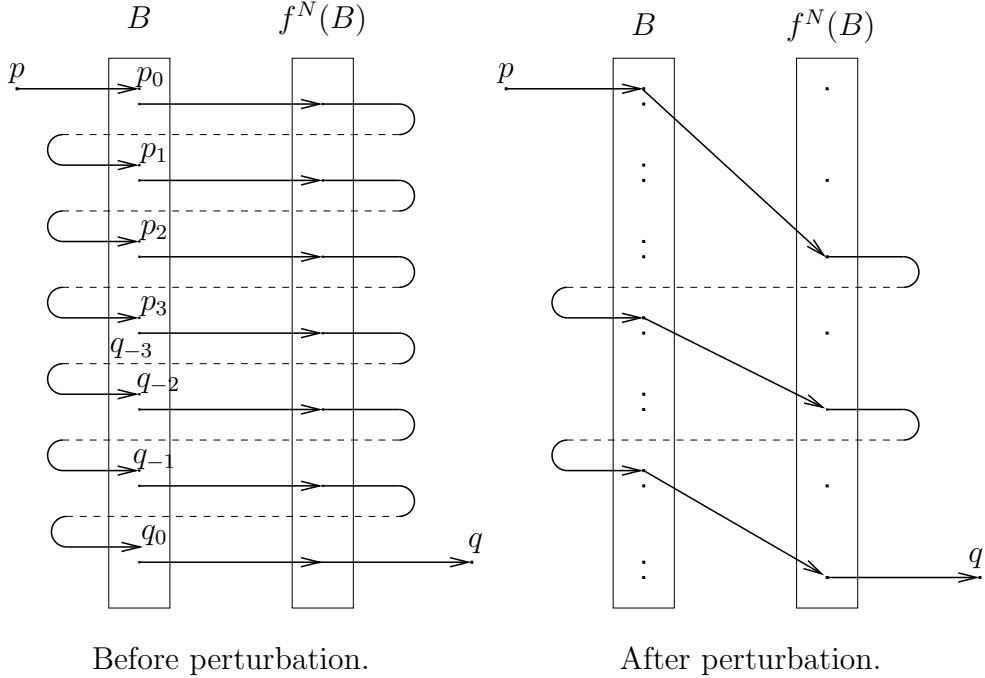}
\end{center}
\caption{Combinatorics of the perturbations realized by the connecting lemma. \label{f.combinatorics}}
\end{figure}

The main difficulty is to choose the subsequence $(x_0,y_0,x_1,y_1,\dots,x_\ell,y_\ell)$.
It is not built directly: we first introduce (section \ref{ss.tiled})
an intermediary sequence $(x'_0,y'_0,x'_1,y'_1,\dots,x'_{\ell'},y'_{\ell'})$ by cleaning up the points in the regions
where there are too much accumulations. We then select a second time (section~\ref{ss.shortcuts})
so that the perturbations associated to each pair
$(x_i,y_i)$ have disjoint supports.

\subsubsectionruninhead{Tiled cubes: first selection.}\label{ss.tiled}
In order to select the points, it is more convenient to replace the euclidian ball $B$ by a square
\footnote{We one replaces a ball by a square, one should specify the orientation of the axes. One chooses
the axis in the directions of the axis of the ellipsis $B$. In other terms,
the square will be viewed in the initial metric as a rectangle and not as a parallelogram.}
(for a metric $\dd'$ given by theorem~\ref{t.pugh}). Viewed with the initial metric, the set $B$ is a rectangle.
We then tile it as pictured on figure~\ref{f.tiled}.
This tiling allows us to decide when too much points of the cloud of returns accumulate in a region of the cube $B$.
\begin{figure}[ht]
\begin{center}
\input{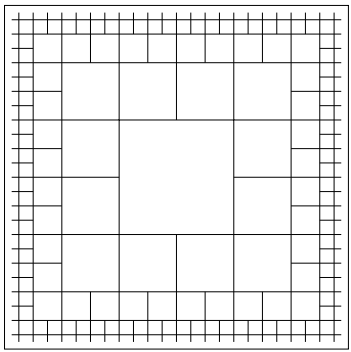}
\end{center}
\caption{A tiled cube. \label{f.tiled}}
\end{figure}

We select a first subsequence $(x'_0,y'_0,x'_1,y'_1,\dots,x'_{\ell'},y'_{\ell'})$
from $(p_0,\dots,p_r,q_{-s},\dots,q_0)$ so that each tile of the square $B$ contains at most one pair
$(x'_i,y'_i)$. This can be done by induction: once $(x'_i,y'_i)$ has been defined, one chooses $x'_{i+1}$
as the first return of $y'_i$ to $B$. The point $y'_{i+1}$ is then the last point in
$\{p_0,\dots,p_r,q_{-s},\dots,q_0\}$ which belong to the tile of $B$ which contains $x'_{i+1}$. 
The only assumption used here is that $p_r$ and $q_{-s}$ are close enough to $z$ so that they belong to the
same central tile of $B$. The picture after this first selection is represented at figure~\ref{f.selection1}.
\begin{figure}[ht]
\begin{center}
\input{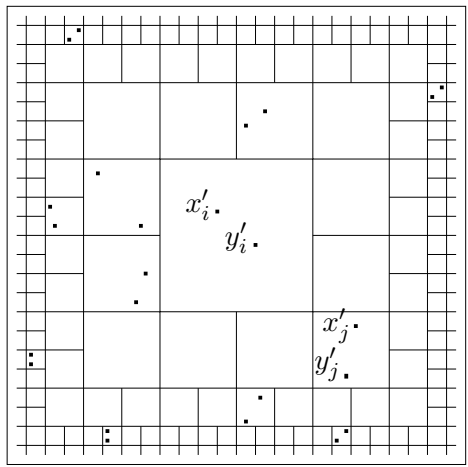}
\end{center}
\caption{The first selection. \label{f.selection1}}
\end{figure}

This first selection is not sufficient for the connecting lemma. Indeed,
if one applies Pugh's theorem~\ref{t.pugh} to define some perturbations $g'_i$ such that
$(g'_i)^N(x_i)=f^N(y_i)$, the supports of these perturbations may overlap: if $x_i$ and $y_i$
belong to a tile $T$ of $B$, the support of the perturbation will be contained in the rectangle $\widehat T$ (obtained
from $T$ by an homothety of ratio $3/2$) and in the $N-1$ first iterates of $\widehat T$. The problem here appears if one has to
perturb in two adjacent tiles, since the two perturbations will conflict.
Hence, we can not compose in general these perturbations $g'_i$ and
we before need to select a second sequence from the cloud
$\{x'_0,y'_0,x'_1,y'_1,\dots,x'_{\ell'},y'_{\ell'}\}$.

\subsubsectionruninhead{Shortcuts: second selection.}\label{ss.shortcuts}
In order to explain how to handle the conflicts described before, let us consider the case where the support
of two perturbations $g'_i$ and $g'_j$ defined above overlap
(they may not overlap in the cube $B$ but in the image $f^k(B)$ of $B$ for some $k\in \{0,\dots,N-1\}$).
This means that the points $\{x'_i,y'_i\}$ (in a tile $T_i$)
and the points $\{x'_j,y'_j\}$ (in a tile $T_j$) have their images by $f^{k}$ close.
We also will assume that $i<j$.

The idea to solve the conflict is two replace the two perturbations which send $x'_i$ and on $f^N(y'_i)$ and $x'_j$
on $f^N(y'_j)$ respectively, by a single perturbation which sends $x'_i$ on $f^N(y'_j)$ (figure~\ref{f.shortcut}).
After this construction, we erase the intermediary points $\{x'_k,y'_k\}$ with $k\in \{i+1,\dots,j-1\}$ from the sequence
$(x'_0,y'_0,x'_1,y'_1,\dots,x'_{\ell'},y'_{\ell'})$ and get a new pseudo-orbit from $p$ to $q$.
In other words, we realized a shortcut in the pseudo-orbit that connects $p$ to $q$.
\begin{figure}[ht]
\begin{center}
\input{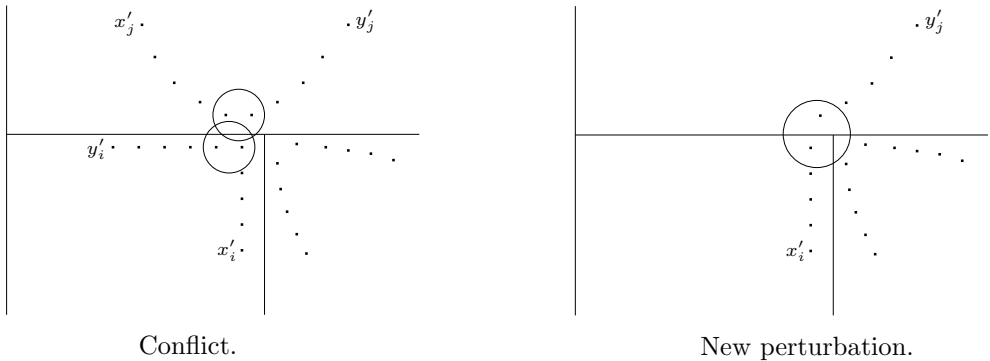}
\end{center}
\caption{A shortcut. \label{f.shortcut}}
\end{figure}

Recall that the supports of the perturbations $g'_i$ and $g'_j$ should be very small in comparison to the iterates of
the tiles $T_i$ and $T_j$ (remark~\ref{r.small-support}) hence, we get that the tiles $T_i$ and $T_j$ should be adjacent and
that the support of the new perturbation remains small in comparison to these tiles.

We then continue in this way in order to solve all the conflicts. We note that each time one solve a conflict,
one gets a new perturbation whose support is a little bit larger. Hence, it could meet the support of another
perturbation $g'_k$ and one should solve a new conflict. One may wonder if the number of conflicts one should consider
starting from an initial pair $(x'_i,y'_j)$ can be arbitrarily large so that the support of the final perturbation may
become huge in comparison to the size of the initial tile $T_i$.

This is not the case: we control a priori the size of all the perturbations so that all the conflicts that can occur
happen with tiles that are adjacent to the initial tile $T_i$. Since the geometry of the tiling is bounded, the
number of tiles that are adjacent to $T_i$ is bounded (by $12$) and we know that we will have to solve at most $12$
conflicts. If one chooses (in Pugh's theorem) the supports of the perturbations $(g'_k)$ very small in comparison
to the size of the corresponding tiles as it is allowed by remark~\ref{r.small-support}, the perturbation that we will get after
solving $12$ conflicts will remains small in comparison to the initial tile $T_i$.
Hence, its support can not reach any new tile
(the other tiles are not tangent to $T_i$ and consequently are far from the support of the perturbation we obtained,
see figure~\ref{f.conflict}).
\begin{figure}[ht]
\begin{center}
\input{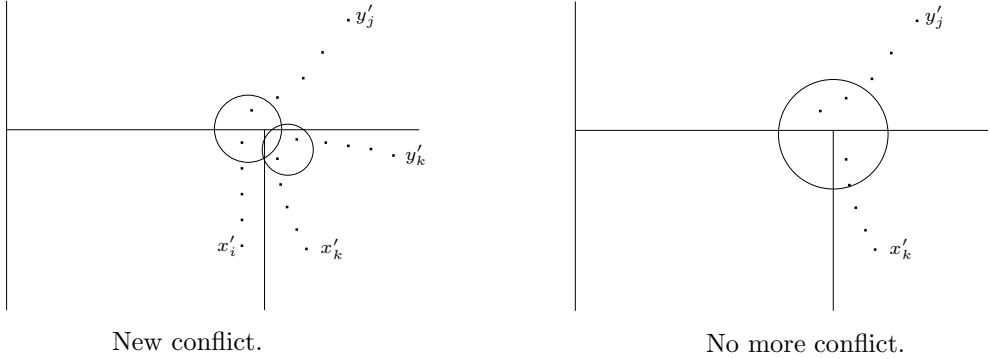}
\end{center}
\caption{The number of conflicts remains bounded. \label{f.conflict}}
\end{figure}
After solving all the conflicts, we get a new pseudo-orbit, a new sequence
$(x_0,y_0,x_1,y_1,\dots,x_\ell,y_\ell)$ and a collection of perturbations $(g_i)$
whose supports are pairwise disjoint and such that, for each $i$, we have
$g^N_i(x_i)=y_i$.

\subsubsectionruninhead{Conclusion.}\label{ss.conclusion-connecting}
The proof gives a better statement than theorem~\ref{t.connecting}. The tiled cube $B$ and the integer $N$ are built
independently from the orbits of the points $p$ and $q$. The only assumption we used was that
$p$ and $q$ have two returns $f^{n(p)}(p)$ and $f^{-n(q)}(q)$ in a same tile of $B$ (for example the central tile
of the cube).

\begin{theo}[Connecting lemma, $2^\text{nd}$ version]
Let $f$ be a $C^1$-diffeomorphism and $\cU$ a neighborhood of $f$ in $\diff^1_v$. For any point $z$ that is not periodic,
there is an integer $N$ such that for any neighborhood $U$ of $z$, there exists a smaller neighborhood $V$
which has the following property:

If $p$ and $q$ are two points outside $U\cup f(U)\cup\dots\cup f^{N}(U)$ that have some iterates
$f^{n(p)}(p)$ and $f^{-n(q)}(q)$ inside $V$, then, there is a perturbation $g\in \cU$ of $f$
with support in $U\cup f(U)\cup\dots\cup f^{N-1}(U)$ and an integer $n\geq 1$ which satisfy $g^n(p)=q$.
\end{theo}
This is this statement that was used to get theorem~\ref{t.horseshoe}.
We prove it by choosing the square $B$ inside $U$; the neighborhood $V$ is the central tile of $B$.
\vspace{1cm}

\newpage
\section{Global perturbations in conservative dynamics}\label{s.global}
In this part, we explain the proof of the connecting lemma for pseudo-orbits.
\begin{theo}\label{t.pseudo-connecting2}
Let $f\in\diff^1_v$ be a diffeomorphism such that for any $\tau\geq 1$, the periodic orbits of period
less than $\tau$ are finite. Then, for any $p,q\in M$, there exists a $C^1$-small perturbation $g\in\diff^1_v$ of $f$
and an integer $n\geq 1$ such that $g^n(p)=q$.
\end{theo}
The proof we give here is due to Arnaud-Bonatti-Crovisier but is slightly different from the arguments
in~\cite{BC,arnaud-bonatti-crovisier}.

\subsection{Introduction}
It is clear that one can not hope to connect a pseudo-orbit into an orbit by performing only a local perturbation:
we shall use several perturbations, given by Hayashi's connecting lemma. Provided they have
disjoint supports, the composed perturbation remains small as it is explained at section~\ref{ss.composition}.

The proof will use two main ingredients:
\begin{enumerate}
\item we will consider (section~\ref{s.perturbation-domains}) a generalization of the tiled cubes given by the proof of the
local connecting lemma (see section~\ref{ss.tiled}): we call them {\em perturbation domains};
\item using strongly the existence of an invariant probability measure with full support (here, a smooth volume),
we build (section~\ref{s.choice}), for
each pair of points $x,y$ that are not periodic, some perturbation domains and a pseudo-orbit that joints $x$ to $y$
and whose jumps are contained in the tiles of the perturbation domains.
\end{enumerate}
The final argument (section~\ref{s.conclusion}) consists in perturbing in each perturbation domain.

\subsection{The perturbation domains}\label{s.perturbation-domains}
One can revisit the proof of Hayashi's connecting lemma and obtain a more general statement.
Let us recall the main ideas. We consider a diffeomorphism $f\in\diff^1_v$ and a neighborhood $\cU$ of $f$.
To any point $z\in M$, which is not periodic, one can associates:
\begin{itemize}
\item an integer $N\geq 1$,
\item a small tiled cube $B$, as pictured on figure~\ref{f.tiled} and disjoint from its $N-1$ first iterates,
\end{itemize}
such that the following property is satisfied:
\begin{description}
\item[\quad\quad] \it Let $p$ be a point whose forward orbit meets the central tile $T_c$ of the tiling of $B$,
and $q$ a point whose backward orbit meets $T_c$.
Then, there exists a perturbation $g\in \cU$ of $f$ with support in $B\cup  f(B)\cup\dots\cup f^{N-1}(B)$
and an integer $m\geq 1$ such that $g^m(p)=q$.
\end{description}
We also recall that the existence of the cube $B$ was given by Pugh's perturbation lemma (theorem~\ref{t.pugh}, see also section~\ref{ss.tiled}).

We now explain how to generalize the assumptions of Hayashi's connecting lemma.

\subsubsectionruninhead{More jumps.}
The connecting lemma allows to connect an orbit to another one provided that the concatenation of
these two segments of orbit is a pseudo-orbit whose unique jump is moreover contained in the cental tile $T_c$ of $B$.
However, the central tile $T_c$ may be replaced by any other tile of $B$, and the proof of the connecting lemma, as explained
at section~\ref{s.connecting}, can also deal with pseudo-orbits $(z_k)$ with any number of jumps. The only assumption that should
be required is that ``{\it the jumps of the pseudo-orbit are contained in the tiles of $B$}".
This means that for each $k$, if $f(z_k)\neq z_{k+1}$
then $f(z_k)$ and $z_{k+1}$ are contained in a same tile.
(Note that for proving the connecting lemma, we already considered this kind of pseudo-orbit: in the first step of the proof
we cleaned up the cloud of returns of the two orbits and obtained, after a first selection, a pseudo-orbit whose jumps
are contained in the tiles of the cube $B$.)

Sometimes it is useful to allow also jumps that are slightly larger than the tiles of $B$:
we introduce the \emph{enlarged tiles $\hat T$ of $B$} as the squares, having the same centers and the same axis as the tiles $T$ of $B$,
but with sizes $1+\frac{1}{10}$ times larger.
Then, we will say that a pseudo-orbit $(z_k)$ \emph{respects the tiling of $B$} if for any $k$
such that $f(z_k)\neq z_{k+1}$, there exists an enlarged tile $\hat T$ which contains both $f(z_k)$ and $z_{k+1}$.

We now claim that the following property is satisfied by the cubes $B$ given by the connecting lemma (figure~\ref{f.prop-P}):
\begin{description}
\item[(P)\quad]\it For any pseudo-orbit $(z_0,\dots,z_n)$ which respects the tiling of $B$,
there exists a perturbation $g\in \cU$ of $f$ with support in $B\cup  f(B)\cup\dots\cup f^{N-1}(B)$
and an integer $m\in \{1,\dots,n\}$ such that $g^m(z_0)=z_n$.
\end{description}
\begin{figure}[ht]
\begin{center}
\input{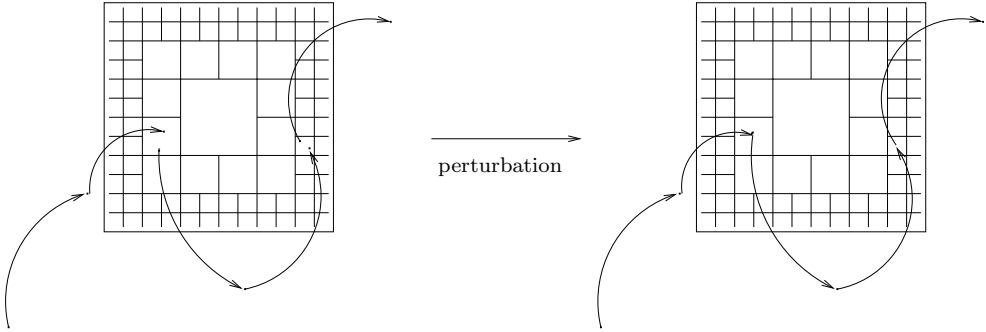}
\end{center}
\caption{Perturbation in a tiled cube (property (P)). \label{f.prop-P}}
\end{figure}

Property (P) is obtained easily from the arguments of section~\ref{s.connecting}:
by a first selection, one can assume that each enlarged tile of $B$ is associated at most to one jump of the pseudo-orbit.
One then removes all the jumps, using the elementary perturbation lemma (proposition~\ref{p.elementary}) for each enlarged tile that contains
a jump. One has to solve the conflicts if necessary:
the crucial argument is that the number of conflicts that are related to a jump is a priori bounded by $12$:
This comes from the following facts:
\begin{itemize}
\item if two enlarged tiles $\hat T$ and $\hat T'$ intersect, then the tiles $T$ and $T'$ are adjacent,
\item the number of tiles $T'$ adjacent to a tile $T$ is bounded by $12$.
\end{itemize}

\subsubsectionruninhead{Tiled domains.}
The cube $B$ which supports the perturbations is obtained from the metric $\dd'$ given by Pugh's theorem~\ref{t.pugh}.
More precisely, there is a local chart $\psi_z\colon U_z\to \RR^2$ defined on a small neighborhood $U_z$ of $z$
such that Pugh's metric $\dd'$ is the pull back by $\psi_z$ of the standard Euclidian metric on $\RR^2$.
In the proof of the connecting lemma, we did not use however the fact that $B$ was a cube; we just needed the bounded
geometry of the tiling: the image $B_0$ of $B$ by $\psi_z$ is a tiled open set of $\RR^2$ which has the following properties:
\begin{description}
\item[(T1)\quad] The tiles of $B_0$ are squares of $\RR^2$ that are oriented along the canonical axis of $\RR^2$.
\item[(T2)\quad] Let us associate to any tile $T_0$, the enlarged cube $\Hat{\Hat{T_0}}$ with the same center but obtained
from $T_0$ by scaling by a ratio $1+\frac{1}{5}$. Then, any two tiles $T_0$ and $T'_0$ are adjacent once
the enlarged cubes $\Hat{\Hat{T_0}}$, $\Hat{\Hat{T_0'}}$ intersect.
\item[(T3)\quad] There exits a uniform constant (here 12) that bounds the number of tiles $T'_0$ that are adjacent to a tile $T_0$ of $B_0$. 
\end{description}
Any tiled open set $B\subset U_z$ which satisfies these two properties will be called
a {\em tiled open set for the chart $(\psi_z,U_z)$} (figure~\ref{f.tiled-domain}).
\begin{figure}[ht]
\begin{center}
\input{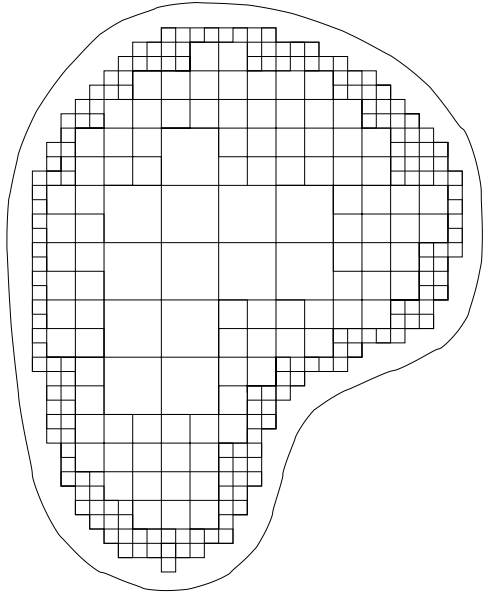}
\end{center}
\caption{A tiled open set. \label{f.tiled-domain}}
\end{figure}

One ends this paragraph by showing how to tile any open set (figure~\ref{f.tiling-domain}):
\begin{prop}\label{p.how-tile}
Any open set $\tilde B$ of $\RR^2$ admits a tiling by squares which satisfies properties (T1), (T2) and (T3).
\end{prop}
\begin{proof}
We introduce the standard tilings $\cT_n$ of $\RR²$ by squares of size $2^{-n}$.
Let $U$ be any open subset of $\RR^2$. The announced tiling $\cT$ is built from the standard tiling $\cT_n$
by deciding inductively what are the tiles of $\cT_n$ that belong to $\cT$: a tile $T\in\cT_n$ will belong to $\cT$ if
\begin{itemize}
\item $T$ and all the tiles in $\cT_n$ that are adjacent to $T$ are contained in $U$,
\item $T$ is not contained in the sub-domain of $U$ tiled at the previous steps by tiles of lower tiling $\cT_m$, $m<n$.
\end{itemize}
\end{proof}
\begin{figure}[ht]
\begin{center}
\input{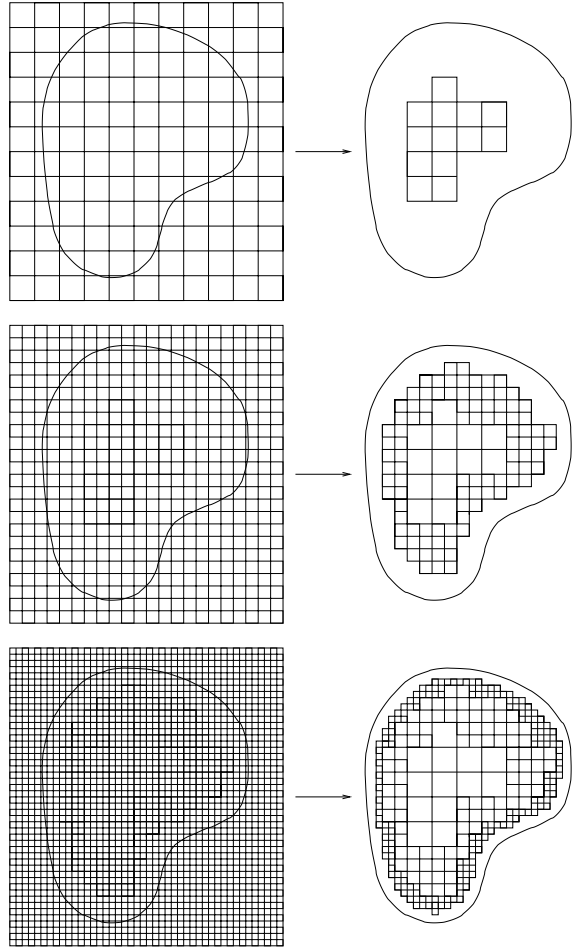}
\end{center}
\caption{How to tile an open set. \label{f.tiling-domain}}
\end{figure}

\subsubsectionruninhead{An improved connecting lemma}
We mention two other technical improvements:
up to here, we have assumed that the point $z$ where we defined the tiled cube or open set is not periodic.
This was necessary, since we need to spread the perturbations in the time during $N$ iterates.
However, the construction of Pugh's metric $\dd'$ around $z$, or equivalently of the chart $\psi_z\colon U_z\to\RR^2$
does not require that $z$ is non-periodic: we only have to consider the sequence of derivatives $\D_z f^n$ at $z$ in order
to analyze what are the directions mostly contracted or expanded by the dynamics.
Then, any tiled cube (or tiled open set) contained in $U_z$ and disjoint from its $N-1$ first iterates will satisfy property (P).
This remark is important since it allows by compactness to cover $M$ by finitely many charts $\psi_z$.
Hence, it is now clear that the integer $N$ depends only on $(f,\cU)$ and not on the choice of the point $z$.

We sum up all the previous remarks by saying that the connecting lemma asserts the existence of tiled open sets which satisfy condition (P):
these sets will be called {\em perturbation domain}.
\begin{theo}[Connecting lemma, $3^\text{rd}$ version]\label{t.connecting3}
For any diffeomorphism $f$ and any neighborhood $\cU$ of $f$ in $\diff^1_v$, there exists an integer $N\geq 1$ and,
at any point $z\in M$, there is a chart $\psi_z\colon U_z\to \RR^2$ such that any tiled open set $B\subset U_z$ for the chart $(\psi_z,U_z)$,
that is disjoint from its $N-1$ first iterates, satisfies property (P).
\end{theo}

\subsubsectionruninhead{Additional remark}
One can state theorem~\ref{t.connecting3} without using tilings. This uses the
idea of the proof of the closing lemma: once Pugh's metric $\dd'$ has been given in a chart $U_z$, one considers the ``hyperbolic distance"
of $B$ defined for pair of points $z,z'\in B$ as $$\dd^B(z,z')=\frac{\dd'(z,z')}{\inf(\dd'(z,\partial B),\dd'(z',\partial B))}.$$
One extends it to the pair of points in $M$: if $z$ or $z'$ does not belong to $B$, then $\dd^B(z,z')=\infty$ unless $z=z'$ (then $\dd^B(z,z')=0$).
If $B$ is tiled as it is described in the proof of proposition~\ref{p.how-tile} and if $\dd^B(z,z')$ is smaller than $\frac{1}{10}$,
then $z$ and $z'$ are contained in a same enlarged tile of $B$ or are equal. This implies the following statement of the connecting lemma:

\begin{theo}[Connecting lemma, $4^\text{th}$ version]\label{t.conecting4}
For any pair $(f,\cU)$, there exists an integer $N\geq 1$ and, at any point $z\in M$, there exists an open neighborhood $U_z$ endowed with
a metric $\dd'_z$ such that any open set $B\subset U_z$ that is disjoint from its $N-1$ first iterates satisfies the following property:

if two points $p,q\in M$ may be jointed by a pseudo-orbit of $f$ whose jumps are smaller than $\frac{1}{10}$ for the hyperbolic distance $\dd^B$ of $B$ associated to $\dd'$,
then there exists a perturbation $g\in \cU$ with support in $B\cup f(B)\cup\dots\cup f^{N-1}(B)$ and an integer $m\geq 1$ such that $g^m(p)=q$.
\end{theo}

\subsection{Choice of the perturbation domains}\label{s.choice}
In order to prove theorem~\ref{t.pseudo-connecting2}, we consider a neighborhood $\cU$ of $f$ (which has the composition property, section~\ref{ss.composition})
and we apply the connecting lemma (theorem~\ref{t.connecting3}) so that we get
\begin{itemize}
\item an integer $N$,
\item a finite covering of $M$ by charts $\psi_z\colon U_z\to \RR^2$.
\end{itemize}
We consider two points $p,q\in M$ that one wants to connect by a same orbit.
One difficulty is that the perturbation domains given by theorem~\ref{t.connecting3} should be disjoint from their $N-1$
first iterates so that they cannot contain any periodic point of small period. This motivates the following argument.

We denote by $\Sigma_{2N}$ the set of points that are periodic of period less or equal to $2N$.
It is finite by assumption. It is always possible to assume that $p$ and $q$ are not periodic
(and not in $\Sigma_{2N}$): if this is
not the case, one chooses $p'$ and $q'$ not periodic and close to $p$ and $q$ (using the fact that the set of
periodic points is only countable), one then realizes the perturbation as explained in the section below and connects $p'$ to $q'$; one ends by a little
conjugacy of the perturbed map in order to move $p'$ on $p$ and $q'$ on $q$.

\subsubsectionruninhead{Connection by segments of orbits and paths.}
In a first step, we build a ``geometrical pseudo-orbit" that joints the points $p$ to $q$: the jumps
are obtained as translations along small paths which are pairwise disjoint.
\begin{prop} There are a sequence of points $(x_0,\dots,x_s)$ in $M$, a sequence of integers
$(n_0,\dots,n_{s-1})$, and a sequence of paths $(\gamma_0,\dots,\gamma_{s-1})$ in $M$ such that:
\begin{itemize}
\item we have $x_0=p$ and $x_s=q$;
\item for each $k\in \{0,\dots,s-1\}$, the points $f^{n_k}(x_k)$ and $x_{k+1}$ are the endpoints of the path $\gamma_k$;
\item each path $\gamma_k$ is small and contained in some domain $U_z$;
\item all the paths $f^i(\gamma_k)$ with $i\in \{0,\dots,N-1\}$ and $k\in \{0,\dots,s-1\}$ are disjoint.
\end{itemize}
\end{prop}
In particular the points $p$ and $q$ are connected by a pseudo-orbit of the form
$$(p=x_0, f(x_0),\dots,f^{n_0-1}(x_0), x_1,f(x_1),\dots, f^{n_1-1}(x_1),\dots,
x_{s-1}, f^{n_{s-1}-1}(x_{s-1}),x_s=q).$$
\begin{proof}
The proof is very similar to the proof of proposition~\ref{p.chainrecurrence}:
since $p$ and $q$ are not periodic, one chooses a large compact set $K$ which contains in its interior $p$ and $q$ and
which is disjoint from the compact set $\Sigma_{2N}$. One may also assume that the interior of $K$ is connected.

There is a pseudo-orbit from $p$ to $q$ whose jumps are contained in the interior of $K$ and
arbitrarily small (in particular each jump is contained in a domain $U_z$). Hence, for any two consecutive
points $x$ and $y$ of the pseudo-orbit, there is a simple path $\sigma$ that connects $f(x)$ to $y$.
One can choose $\sigma$ small and contained in the interior of $K$ so that $\sigma$ is disjoint from its $2N$
first iterates and contained in some open set $U_z$. We get a family of path $(\sigma_k)$.
This gives all the required properties but the last one (a path $\sigma$ may have some iterate
$f^\ell(\sigma)$ with $\ell\in \{0,\dots,N\}$ that intersects another path $\sigma'$).

We now suppress the intersections between the paths. This is done inductively: one considers the smallest $i$
such that one iterate of $\sigma_i$ intersects one iterate of some other path $\sigma_j$.
One considers also the largest possible $j$. One then modifies the pseudo-orbit and the sequence $(\sigma_k)$:
we suppress the intersections between $\sigma_i$ and $\sigma_j$, we do not create any new intersection and
we do not use anymore the paths $\sigma_k$ with $k\in \{i+1,\dots,j-1\}$.
Hence, after a finite number of modifications, we have removed all the intersections and we get the announced sequences.

Let us explain how to handle with the intersections
(we will assume for instance that $f^\ell(\sigma_i)$ intersect $\sigma_j$):
the path $\sigma_i$ connects the point $f^{n_i}(x_i)$ to $x_{i+1}$ and
the path $\sigma_j$ connects the two points $f^{n_i}(x_j)$ and $x_{j+1}$.
One considers the path $\sigma'_i\subset \sigma_i$ which connects $f^{n_i}(x_i)$ to the first intersection
point $x'_{i+1}$ of $\sigma_i$  with $f^{-\ell}(\sigma_j)$.
We introduce also the path $\tilde \sigma_j\subset \sigma_j$
which connects $f^\ell(x'_{i+1})$ to $x_{j+1}$. Note that $f^\ell(\sigma'_i)$ and $\tilde \sigma_j$
only intersect at their endpoint $f^\ell(x'_{i+1})$.\\
One now considers a return $f^{n'_{i+1}}(x'_{i+1})$
of $f^\ell(x'_{i+1})$ close to $f^\ell(x'_{i+1})$ (by Poincar\'e recurrence theorem, changing a little bit
the point $x'_{i+1}$ if necessary). One may also assume that $f^{n'_{i+1}}(x'_{i+1})$ does not belong to
$f^\ell(\sigma'_{i})$ (perturbing a little bit $\sigma'_{i}$ again). It is then possible to
modify $\tilde \sigma_j$ in a neighborhood of $f^\ell(x'_{i+1})$ so that it connects $f^{n'_{i+1}}(x'_{i+1})$ to
$x_{j+1}$ and is disjoint from $f^\ell(\sigma'_{i})$: we get a path $\sigma'_j$
(see figure~\ref{f.path}).
\begin{figure}[ht]
\begin{center}
\input{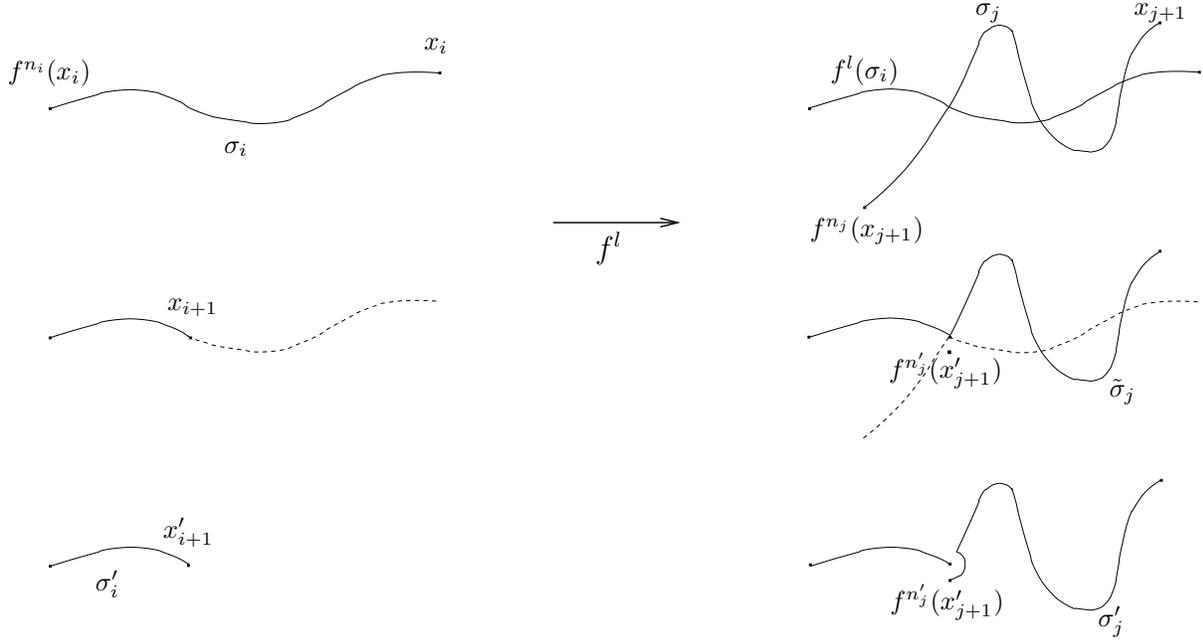}
\end{center}
\caption{Modification of the paths $\sigma_i$ and $\sigma_j$. \label{f.path}}
\end{figure}

By this construction, there is no more intersection between the iterates of $\sigma'_i$ and $\sigma'_j$.
There is also a segment of orbit that connects one endpoint of $\sigma'_i$ to an endpoint of $\sigma'_j$.
Hence, one can forget the intermediary paths $\sigma_k$ with $k\in \{i+1,\dots,j-1\}$.
\end{proof}

\subsubsectionruninhead{Construction of perturbation domains.}\label{ss.construction-domains}
In the second step, we here build:
\begin{enumerate}
\item A finite number of perturbation domains $(B_s)_{s\in S}$.
One requires that all the sets $f^i(B_s)$, with $s\in S$, and $i\in\{0,\dots, N-1\}$ are pairwise disjoint.
\item A pseudo-orbit $(z_0,\dots,z_n)$ between $p$ and $q$ which respects the tiling of the perturbation domains $B_k$.
\end{enumerate}

For this, one first considers the sequences $(x_k)$, $(n_k)$ and $(\gamma_k)$ introduced at the previous section
and that connect $p$ to $q$. One thicks each path $\gamma_k$ as a small open set $B_k$.
As it is the case for the paths $\gamma_k$, one can assume that each $B_k$ is contained in some open set $U_z$
and all the sets $f^i(B_k)$ with $i\in \{0,\dots, N-1\}$ and $s\in S$ are pairwise disjoint.\\
By proposition~\ref{p.how-tile}, one can tile each $B_k$ as a tiled domain of the chart $(\psi_z,U_z)$
and by theorem~\ref{t.connecting3}, it becomes a perturbation domain.\\

It remains to define the pseudo-orbit $(z_0,\dots,z_n)$. Note that the pseudo-orbit
$$(p=x_0, f(x_0),\dots,f^{n_0-1}(x_0), x_1,\dots,x_s=q)$$
has its jumps contained in the domains $B_k$ but maybe not in the tiles of the $B_k$. Therefore, one
will define a longer pseudo-orbit by introducing between each pair of points $(f^{n_i}(x_i), x_{i+1})$
a segment of pseudo-orbit which has its jumps in the enlarged tiles of the domains $(B_k)$.\\
The argument is again the same as in the proof of proposition~\ref{p.chainrecurrence}:
since $\gamma_k$ is connected, one can build a pseudo-orbit from $f^{n_i}(x_i)$ to $x_{i+1}$
(using the recurrence of almost every point of $M$ by $f$) and having jumps
arbitrarily small and all contained in any neighborhood of $\gamma_k$.
In particular, if the size of the jumps and the distance from each jump to $\gamma_k$ are
small with respect to the tiles of $B_k$ that meet $\gamma_k$, then, the pseudo-orbit respect the tiling of $B_k$.
This gives the pseudo-orbit $(z_k)$.

\subsection{Conclusion: proof of theorem~\ref{t.pseudo-connecting2}}\label{s.conclusion}
We have built some domains $(B_k)$ and a pseudo-orbit $(z_i)$ from $p$ to $q$
at section~\ref{ss.construction-domains}.\\
Using the definition of the perturbation domains (theorem~\ref{t.connecting3}), one may perturb in each $B_k$ in order to remove the
jumps inside (property (P)).
Since the perturbation domains are disjoint, one can perturb independently in each of them, the
final perturbation will stay in $\cU$ by the composition property.\\
Hence, one considers each perturbation domains one after the other and eliminates in this way
all the jumps of the pseudo-orbit $(z_i)$ (figure~\ref{f.suppress}). We get at the end for the perturbed map
a genuine orbit that connects $p$ to $q$ as required.
\begin{figure}[ht]
\begin{center}
\input{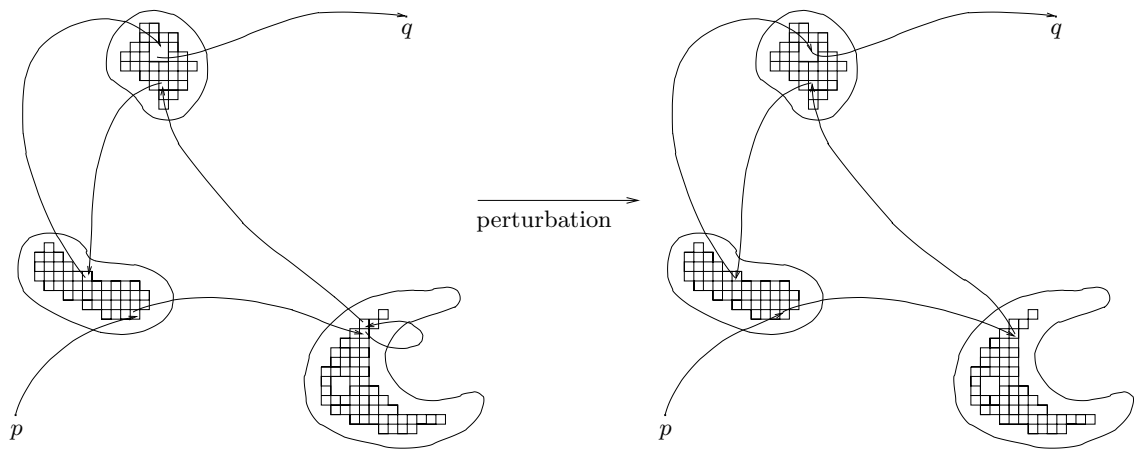}
\end{center}
\caption{Final perturbation. \label{f.suppress}}
\end{figure}

\vskip 1cm
\flushleft{\bf Sylvain Crovisier} \ \  (crovisie@math.univ-paris13.fr)\\

\medskip
\medskip
CNRS - Laboratoire Analyse, G\'eom\'etrie et Applications, UMR 7539,\\
Institut Galil\'ee, Universit\'e Paris 13,
Avenue J.-B. Cl\'ement, 93430 Villetaneuse, France\\


\begin{thebibliography}{ZZZZZ}
\small
\bibitem[A$_1$]{arnaud1}
M.-C. Arnaud,
{\it Le ``closing lemma" en topologie $C\sp 1$,} M\'em. Soc. Math. Fr. {\bf 74} (1998).
\bibitem[A$_2$]{arnaud2}
M.-C. Arnaud,
{\it Cr\'eation de connexions en topologie $C\sp 1$,} Ergod. Th. \& Dynam. Sys. {\bf 21} (2001), 339--381.
\bibitem[ABC]{arnaud-bonatti-crovisier}
M.-C. Arnaud, C. Bonatti and S. Crovisier,
\textit{Dynamiques symplectiques g\'en\'eriques}, preprint IMB {\bf 363} (2004), to appear in Ergod. Th. \& Dynam. Sys.
\bibitem[B]{bochi}
J. Bochi, {\it Genericity of zero Lyapunov exponents,} Ergod. Th. \& Dynam. Sys. {\bf 22} (2002), 1667--1696.
\bibitem[BC]{BC}
C. Bonatti and S. Crovisier,
{\em R\'ecurrence et g\'en\'ericit\'e,} Invent. Math. \textbf{158} (2004), 33-104.
\bibitem[BDP]{bonatti-diaz-pujals}
C. Bonatti, L. D\'iaz and E. Pujals,
{\em A $\cC^1$-generic dichotomy for diffeomorphisms: weak forms of hyperbolicicity or infinitely many sinks or sources,} Ann. Math. {\bf 158} (2003), 355--418.
\bibitem[C]{crovisier}
S. Crovisier, {\em Periodic orbits and chain transitive sets of $C^1$-diffeomorphisms,} preprint IMB {\bf 368} (2004).
\bibitem[F$_1$]{franks1}
J. Franks, {\it Anosov diffeomorphisms,} Global Analysis (Berkeley, 1968), Amer. Math. Soc., Providence (1970), 61--93.
\bibitem[F$_2$]{franks2}
J. Franks, 
{\it Necessary conditions for stability of diffeomorphisms,} 
Trans. Amer. Math. Soc. {\bf 158} (1971), 301--308.
\bibitem[FL]{franks-lecalvez}
J. Franks and P. Le Calvez,
{\it Regions of instability for non-twist maps,}
Ergod. Th. \& Dynam. Sys. {\bf 23} (2003), 111--141.
\bibitem[Ha]{hayashi}
S. Hayashi, {\em Connecting invariant manifolds and the  solution of the $C^1$-stability and $\Omega$-stability conjectures for flows,} Ann. of Math. {\bf 145} (1997), 81--137 and Ann. of Math. {\bf 150} (1999), 353--356.
\bibitem[He]{herman}
M. Herman,
{\em Sur les courbes invariantes par les diff\'eomorphismes de l'anneau,} Ast\'erisque {\bf 103--104}, (1983).
\bibitem[Hu]{hubbard}
J. Hubbard,
{\it The forced damped pendulum: chaos, complication and control,} 
Amer. Math. Monthly {\bf 106} (1999), 741--758.
\bibitem[Man]{mane}
R. Ma\~n\'e, \textit{An ergodic closing lemma,} Ann. of Math. {\bf 116} (1982),
503--540.
\bibitem[Mat]{mather}
J. Mather,
{\it Invariant subsets for area preserving homeomorphisms of surfaces,}
Mathematical analysis and applications,  
Adv. in Math. Suppl. Stud. {\bf 7b}, 
Academic Press, New York-London (1981), 531--562.
\bibitem[Mo]{moser}
J. Moser, {\em Stable and random motions in dynamical systems,} Annals of mathematics studies {\bf 77}, Princeton university Press (1973).
\bibitem[N]{newhouse}
S. Newhouse, {\it Quasi-elliptic periodic points in conservative dynamical systems,}
Amer. J. Math. {\bf 99} (1977), 1061--1087.
\bibitem[O]{oliveira}
F. Oliveira, {\it On the generic existence of homoclinic points,} Ergod. Th. \& Dynam. Sys. {\bf 7} (1987), 567--595.
\bibitem[OU]{oxtoby-ulam}
J. Oxtoby and S. Ulam,
{\it Measure-preserving homeomorphisms and metrical transitivity,}
Ann. of Math. {\bf 42} (1941), 874--920.
\bibitem[P]{pixton}
D. Pixton, {\it Planar homoclinic points,} J. Diff. Eqns {\bf 44} (1982), 365--382.
\bibitem[Pu$_1$]{pugh1}
C. Pugh,
\textit{The closing lemma,} Amer. J. Math. {\bf 89} (1967), 956--1009.
\bibitem[Pu$_2$]{pugh2}
C. Pugh,
\textit{An improved closing lemma and a general density theorem,} Amer. J. Math. {\bf 89} (1967), 1010--1021.
\bibitem[PR]{pugh-robinson}
C. Pugh and C. Robinson,
\textit{The $C^1$-closing lemma, including hamiltonians,}
Ergod. Th. \& Dynam. Sys. \textbf{3} (1983), 261--314.
\bibitem[R$_1$]{robinson1}
C. Robinson,
{\it Generic properties of conservative systems, I and II,}
Amer. J. Math. {\bf 92} (1970), 562--603 and 897--906.
\bibitem[R$_2$]{robinson2}
C. Robinson,
{\it Closing stable and unstable manifolds on the two sphere,}
Proc. Amer. Math. Soc. {\bf 41} (1973), 299--303.
\bibitem[T]{takens}
F. Takens, {\it Homoclinic points in conservative systems,} Invent. Math. {\bf 18} (1972), 267--292.
\bibitem[Y]{yoccoz}
J.-C. Yoccoz,
{\it Travaux de Herman sur les tores invariants,} S\'eminaire Bourbaki 784,
Ast\'erisque {\bf 206} (1992), 311--344.
\bibitem[WX1]{wen-xia0}
L. Wen and Z. Xia,
{\it A basic $C¹$-perturbation theorem,} J. differential equations \textbf{154} (1999), 267--283.
\bibitem[WX2]{wen-xia}
L. Wen and Z. Xia,
\textit{$C^1$ connecting lemmas,}
Trans. Amer. Math. Soc. \textbf{352} (2000), 5213--5230.
\bibitem[Z$_1$]{zehnder1}
E. Zehnder,
{\it Homoclinic points near elliptic fixed points,}
Comm. Pure Appl. Math. {\bf 26} (1973), 131--182.
\bibitem[Z$_2$]{zehnder2}
E. Zehnder,
{\it Note on smoothing symplectic and volume-preserving diffeomorphisms,} Geometry and topology (Rio de Janeiro, 1976),
Lecture Notes in Math. {\bf 597}, 828--854.
\end{thebibliography}
\end{document}